\documentclass[11pt,a4paper]{article}
\usepackage[dvips]{graphicx}
\usepackage[ansinew]{inputenc}
\usepackage{enumerate}
\usepackage[spanish,english]{babel}
\usepackage{amsthm,amssymb,amsfonts,amsmath}

\newtheorem{theorem}{Theorem}[section]
\newtheorem{lemma}[theorem]{Lemma}
\newtheorem{proposition}[theorem]{Proposition}

\newenvironment{remark}
   {{\parindent=0pt {\bf Remark.}}}
   {\vspace{0.1cm}}

\DeclareMathOperator{\Auto}{Aut} \DeclareMathOperator{\Det}{Det}
\DeclareMathOperator{\Stab}{Stab}

\title{Hypergraphs for computing determining sets of Kneser graphs}

\author{J. Cáceres \thanks{Department of Statistics and Applied Mathematics,
University of Almería, Spain. Email addresses: {\tt
\{jcaceres,mpuertas\}@ual.es}. Partially supported by projects
MTM2008-05866-C03-01, P06-FQM-01649 and PAI FQM-305.} \and D.
Garijo \thanks{Department of Applied Mathematics I, University of
Seville, Spain. Email addresses: {\tt
\{dgarijo,gonzalezh,almar\}@us.es}. Partially supported by
projects MTM2008-05866-C03-01, P06-FQM-01649 and PAI FQM-164.}
\and A. González$^\dag$ \and A. Márquez$^\dag$ \and M. L.
Puertas$^*$}

\begin{document}

\maketitle

\begin{abstract}
A set of vertices $S$ is a \emph{determining set} of a graph $G$
if every automorphism of $G$ is uniquely determined by its action
on $S$. The \emph{determining number} of $G$ is the minimum
cardinality of a determining set of $G$. This paper studies
determining sets of Kneser graphs from a hypergraph perspective.
This new technique lets us compute the determining number of a
wide range of Kneser graphs, concretely $K_{n:k}$ with $n\geq
\frac{k(k+1)}{2}+1$. We also show  its usefulness  by giving
shorter proofs of the characterization of all Kneser graphs with
fixed determining number 2, 3 or 4, going even further to fixed
determining number 5. We finally establish for which Kneser
graphs $K_{n:k}$ the determining number is equal to $n-k$,
answering a question posed by Boutin.

\

\end{abstract}

\section{Introduction}


The \emph{determining number} of a graph $G$ is the minimum
cardinality of a set $S\subset V(G)$ such that the automorphism
group of the graph obtained from $G$ by fixing every vertex in $S$
is trivial. The set $S$ is called a \emph{determining set} of $G$.
Although they were first defined as \emph{fixing sets} by
Harary~\cite{harary0}, we follow the terminology of \cite{boutin}
(see also \cite{albertson2}) since the author develops a study on
Kneser graphs. Specifically,  tight bounds for their determining
numbers are obtained and all Kneser graphs with determining
number $2$, $3$ or $4$ are provided. The main tools used in
\cite{boutin} to find a determining set or to bound a determining
number of a Kneser graph are based on characteristic matrices and
vertex orbits.

The notion of determining set has its origin in the idea of
distinguishing the vertices in a graph, particularly in the
concept of \emph{symmetry breaking} which was introduced by
Albertson and Collins \cite{albertson} and, independently, by
Harary \cite{harary2,harary0}. Symmetry breaking has several
applications; among them those related to the problem of
programming a robot to manipulate objects \cite{lynch}.
Determining sets have been since then widely studied. There
exists by now an extensive literature on this topic. Besides the
above-mentioned references, see for instance
\cite{albertson2,boutin2,erwin,gibbons}.

On the other hand, Kneser's conjecture states that the graph with
all $k-$element subsets of $\{1, \ldots ,n\}$ as vertices and with
edges connecting disjoint sets has chromatic number $n-2k+2$.
Kneser proposed this conjecture in \cite{mkneser} in connection
with a study of quadratic forms. Its first proof by Lovasz
\cite{lovasz} was the beginning of topological combinatorics as a
field of research. Kneser graphs arose then as an interesting
family of graphs to explore, and several topological proofs of
the Kneser's conjecture have been published; among them those of
Bárány \cite{barany} and Sarkaria \cite{sarkaria}. In 2004, it
appeared the first combinatorial proof of this conjecture, due to
Matou$\check{\rm s}$ek \cite{matousek}. Besides, in \cite{furedi}
 extremal problems concerning these graphs are investigated.

This paper addresses a general study of determining sets of Kneser
graphs. Our main contribution is to introduce hypergraphs as a
tool for finding determining sets which is done in Section 3.
Indeed, we prove that every subset of vertices $S$ of a Kneser
graph  $K_{n:k}$ has an associated hypergraph $\mathcal{H}_S$. The
set is determining whenever the hypergraph is $k$-regular and has
either $n$ or $n-1$ edges. We also show that every $k-$regular
simple hypergraph with either $n$ or $n-1$ edges and $n\geq 2k+1$
has an associated determining set of $K_{n:k}$. This
characterization lets us compute the determining number of all
Kneser graphs $K_{n:k}$ verifying that $n\geq
\frac{k(k+1)}{2}+1$, which is a significant advance since the
only exact values obtained previously are for $n=2^r-1$ and
$k=2^{r-1}-1$ (see \cite{boutin} for details).

In Section 4, we list all Kneser graphs with fixed determining
number 2, 3, 4 or 5. In the first three cases, we provide shorter
proofs of those developed in \cite{boutin} in order to show that
hypergraphs play an important role in the study of the determining
number of Kneser graphs. Indeed, our technique lets us go further
to fixed determining number equal to 5.

Section 5 concerns the question of whether there exists an
infinite family of Kneser graphs $K_{n:k}$ with $k\geq 2$ and
determining number $n-k$, which was posed by Boutin in
\cite{boutin}. The answer to this question is given by Theorem
\ref{n-k}; to reach it we use as a main tool our approach of
determining sets by hypergraphs.

We conclude the paper with some remarks and open problems.

\section{Preliminaries}

Graphs in this paper are finite and undirected. The vertex set and
the edge set of a graph $G$ are denoted by $V(G)$ and $E(G)$,
respectively. A \emph{hypergraph} is a generalization of a graph,
where edges can connect any number of vertices. Formally, a
hypergraph $\cal H$ is a pair $(V({\cal H}), E(\cal H))$ where
$V(\cal H)$ is the set of \emph{vertices}, and $E(\cal H)$ is a
set of non-empty subsets of $V(\cal H)$ called \emph{hyperedges}
or simply \emph{edges}. When edges appear only once, the
hypergraph is called \emph{simple}. The \emph{order} of a
hypergraph is the number of its vertices, written as $|V(\cal
H)|$, and the \emph{size} is the number of its edges $|E(\cal
H))|$. A hyperedge containing $m$ vertices is said to be an edge
of size $m$. Thus, given a hypergraph with $n$ vertices there are
edges of size ranging over the set $\{1,\ldots, n\}$. The
\emph{degree} of a vertex $v$ is the number of hyperedges
containing $v$. A hypergraph is called $k$-\emph{regular} if
every vertex has degree $k$. For more terminology we follow
\cite{berge} and \cite{west}. In all the figures in this paper,
the hyperedges of size greater than 2 are illustrated as shadowed
regions.

\subsection{Determining Sets}

An \emph{automorphism} of $G$, $f:V(G)\to V(G)$, is a bijective
mapping such that $$f(u)f(v)\in E(G)\Longleftrightarrow uv\in
E(G).$$ As usual $\Auto(G)$ denotes the automorphism group of $G$.

A subset $S\subseteq V(G)$ is said to be a \emph{determining set}
of $G$ if whenever $g,h\in$ $\Auto(G)$ so that $g(s)=h(s)$ for all
$s\in S$, then $g(v)=h(v)$ for all $v\in V(G)$. The minimum
cardinality of a determining set of $G$, denoted by $\Det(G)$, is
called the \emph{determining number} of $G$.

Note that every graph has a determining set. It suffices to
consider any set containing all but one vertex. Thus, $\Det(G)\le
|V(G)|-1$. The only connected graphs with $\Det(G)=|V(G)|-1$ are
the complete graphs. A graph $G$ with $\Det(G)=0$ is called
\emph{asymmetric} or \emph{rigid graph} (see for instance
\cite{albertson}). In \cite{beineke}, it is proved that almost all
graphs are rigid.

A characterization of determining set is provided by Boutin
in~\cite{boutin} by using the \emph{pointwise stabilizer} of $S$,
that is, for any $S\subseteq V(G)$ $$\Stab(S)=\{g\in\Auto(G) \, |
\, g(v)=v, \forall v\in S\}=\bigcap_{v\in S}\Stab(v).$$

\begin{proposition}{\rm{\cite{boutin}}}\label{estabilizador}
$S\subseteq V(G)$ is a determining set of $G$ if and only if
$\Stab(S)=\{id\}$.
\end{proposition}

There are many graphs whose determining number can be easily
computed. Among them the path $P_n$, the cycle $C_n$ and the
complete graph $K_n$. An extreme is a minimum determining set of
$P_n$ and so ${\rm Det}(P_n)=1$. Any pair of non-antipodal
vertices is a determining set of a cycle, thus $\Det(C_n)=2$. A
minimum determining set of $K_n$ is any set containing all but
one vertex, and hence $\Det(K_n)=n-1$. Example 1 of \cite{boutin}
shows that the determining number of the Petersen graph is equal
to three. Nevertheless, computing the determining number can
require intricate arguments (see for instance \cite{boutin,
boutin2, CGPS}). Figure \ref{productos} shows minimum determining
sets of some Cartesian products of graphs.

\begin{figure}[ht]
\begin{center}

\begin{tabular}{cccc}
\includegraphics[width=0.2\textwidth]{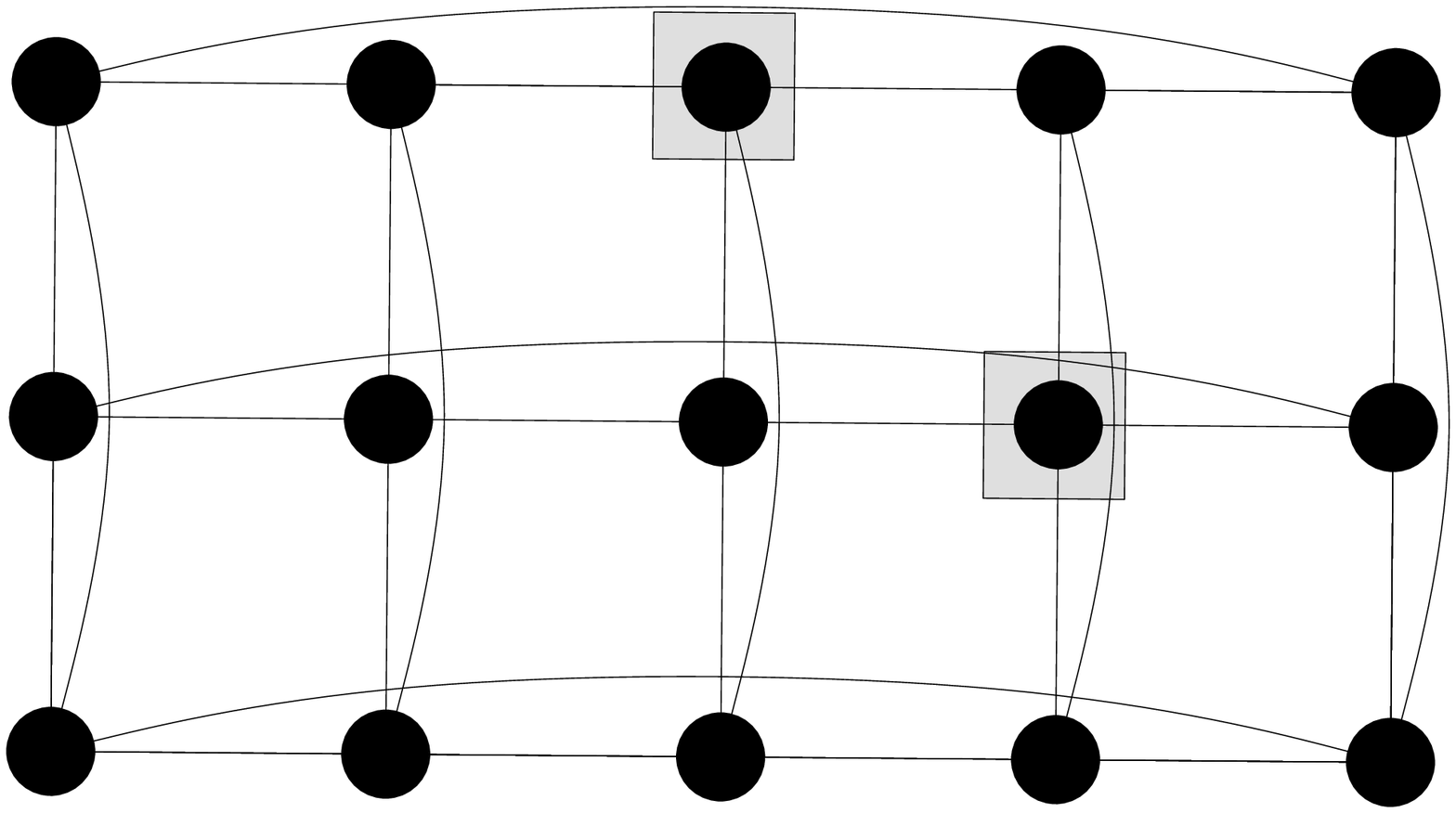}&
 & &
\includegraphics[width=0.3\textwidth]{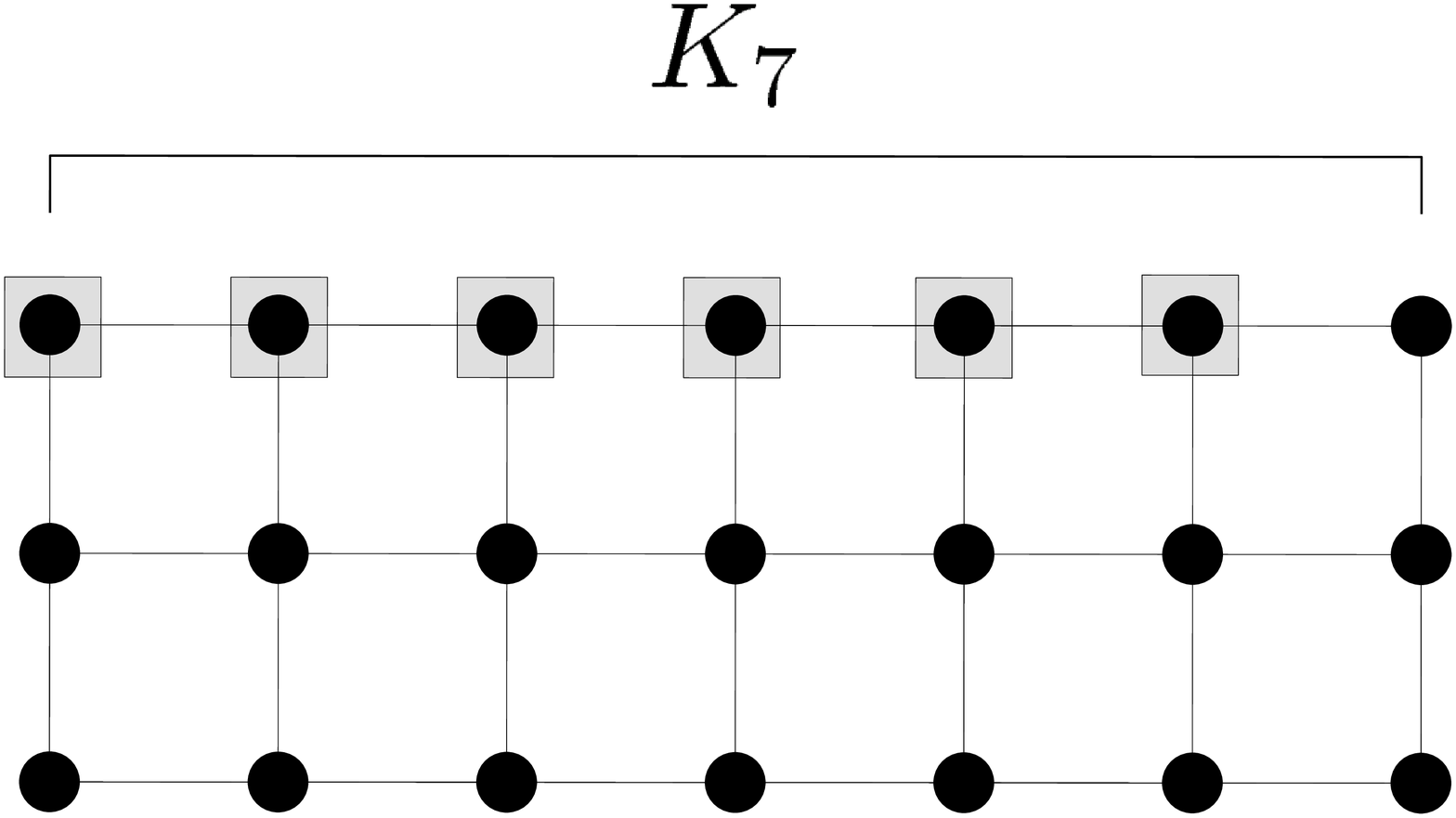}\\ & & \\
(a) & &  &  (b) \\

\end{tabular}
\end{center}
\caption{(a) The squared vertices form a minimum determining set
of $C_5\Box C_3$, (b) The squared vertices form a minimum
determining set of $K_7\Box P_3$. For the sake of clarity, we omit
some edges of the complete graph $K_7$. }\label{productos}
\end{figure}

\subsection{Kneser Graphs}

The Kneser graph $K_{n:k}$ has vertices associated with the
$k-$subsets of the $n-$set $[n]=\{1,\ldots, n\}$ and edges
connecting disjoint sets (see Figure \ref{kneser73}). This family
of graphs is usually considered for $n\geq 2k$ since for $n<2k$
we obtain $\binom{n}{k}$ isolated vertices. Moreover, the case
$n=2k$ gives a set of disconnected edges which is not an
interesting case for computing the determining number that is
equal to $\frac{1}{2}\binom{2k}{k}$, that is, half the number of
vertices. Therefore, throughout this paper we shall assume that
$n>2k$ and vertices will simultaneously be considered as $k-$sets
and as vertices. Thus, the \emph{complementary} of a vertex $W\in
V(K_{n:k})$, written as $W^c$, is the $(n-k)-$subset $[n]\setminus
W$.

\begin{figure}[ht]
\begin{center}

\begin{tabular}{c}
\includegraphics[width=0.4\textwidth]{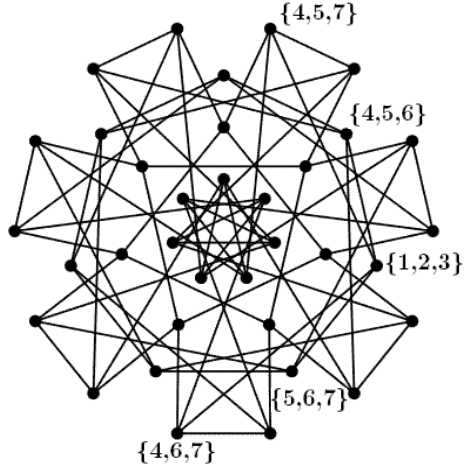}

\end{tabular}
\end{center}
\caption{Kneser graph $K_{7:3}$. The neighbours of $\{1,2,3\}$
are $\{4,5,7\}$, $\{4,5,6\}$, $\{5,6,7\}$ and $\{4,6,7\}$, since
they have no elements in common. }\label{kneser73}
\end{figure}

Boutin in \cite{boutin} provides a first characterization of
determining sets of Kneser graphs which is a key tool also in this
paper.

\begin{lemma}{\rm \cite{boutin}} \label{caractboutin}
The set $\{V_1,\ldots V_r\}$ is a determining set of $K_{n:k}$ if
and only if there exists no pair of distinct elements $a,b\in
[n]$ so that for each $i$ either $\{a,b\}\subseteq V_i$ or
$\{a,b\}\subseteq V_i^c$.
\end{lemma}

Observe that the above lemma implies that every determining set
of $K_{n:k}$ has to contain all the elements of $[n]$ but at most
one.

Lemma \ref{caractboutin} is used in \cite{boutin} to obtain tight
upper and lower bounds for Det$(K_{n:k})$. Concretely, the author
shows that $\log_2(n+1)\leq \Det(K_{n:k})\leq n-k$ and provides
the exact value $\Det(K_{2^r-1:2^{r-1}})=r$.

\section{Computing the determining number of Kneser graphs}

In this section, we characterize determining sets of Kneser graphs
in terms of hypergraphs. This approach is our key tool to compute
the determining number of a wide range of Kneser graphs.

For any set of vertices $S \subseteq V(K_{n:k})$ denote by
$\mathcal{H}_S$ the $k$-regular hypergraph obtained as follows.
The vertex set $V(\mathcal{H}_S)$ is equal to $S$, and two
vertices belong to the same hyperedge whenever they contain  a
common element. When an element of $[n]$ appears only once in the
vertices of $S$, we have a loop in the corresponding vertex of
$\mathcal{H}_S$. Figure~\ref{hypergraph} shows an instance of
hypergraph associated to a set $S\subseteq V(K_{6:3})$.

\begin{figure}[ht]
\begin{center}
\includegraphics[width=0.3\textwidth]{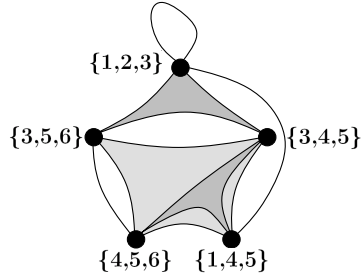} 
\label{hypergraph} \caption{Hypergraph associated to
$S=\{\{1,2,3\},\{3,4,5\},\{1,4,5\},\{4,5,6\},\{3,5,6\}\}$. Number
2 appears only once in $S$ so we have a loop attached to the
vertex $\{1,2,3\}$. There are two edges of size 2 determined by
numbers 1 and 6, two edges of size 3 since numbers 3 and 4 appear
in three vertices of $S$, and one edge of size 4 determined by
number 5.}
\end{center}
\end{figure}

The two following results state that the condition of being
determining set can be captured from the structure of the
associated hypergraph.

\begin{lemma} \label{conditions1}
A vertex set $S$ is a determining set of $K_{n:k}$ with $n\geq
2k+1$ if and only if the associated $k$-regular hypergraph
$\mathcal{H}_S$ is simple and has either $n$ or $n-1$ edges.
\end{lemma}

\begin{proof}
($\Longrightarrow$) Consider a determining set $S$ of $K_{n:k}$
and the associated $k$-regular hypergraph $\mathcal{H}_S$. By
Lemma \ref{caractboutin}, there exists no pair of distinct
elements $a,b\in [n]$ so that for each vertex $V\in S$ either
$\{a,b\}\subseteq V$ or $\{a,b\}\subseteq V^c$. Hence,
$\mathcal{H}_S$ is simple. Indeed, having multiple edges in
$\mathcal{H}_S$ is equivalent to have at least two elements of
$[n]$ in exactly the same vertices of $S$, what implies that they
are not distinguishable by any vertex of $S$.

It remains to prove that $\mathcal{H}_S$ has either $n$ or $n-1$
edges. By Lemma \ref{caractboutin}, the vertices of every
determining set $S$  have to contain all the elements of $[n]$
but at most one.  The result follows since elements of $[n]$ in
$S$ correspond to edges in $\mathcal{H}_S$.

($\Longleftarrow$) Suppose that $\mathcal{H}_S$ is simple and has
either $n$ or $n-1$ edges. Then, for every $a,b\in [n]$ the
corresponding edges in $\mathcal{H}_S$ are different and at most
one element of $[n]$ is contained in no vertex of $S$. By Lemma
\ref{caractboutin} yields the result.
\end{proof}

\begin{lemma}\label{conditions}
For any $k-$regular simple hypergraph $\mathcal{H}$ with either
$n$ or $n-1$ edges and $n\geq 2k+1$, there exists a determining
set $S$ of $K_{n:k}$ such that $\mathcal{H}\cong \mathcal{H}_S$.
\end{lemma}

\begin{proof}
We first label every edge of $\mathcal{H}$ with the elements of
either $[n]$ or $[n-1]$ depending on the number of edges. The
vertices of $\mathcal{H}$ are labeled with the labels of their
incident edges, giving rise to $|V(\mathcal{H})|$ different
$k-$subsets of $[n]$. Take $S$ as the set formed by these
$k-$subsets. Clearly, $\mathcal{H}\cong \mathcal{H}_S$ and by
Lemma \ref{conditions1} the result follows.
\end{proof}

When a determining set $S$ is minimum, Lemma \ref{conditions}
guarantees that it does not exist a $k-$regular, simple
hypergraph $ \mathcal{H}$ with either $n$ or $n-1$ edges and
$|V(\mathcal{H})|<|S|$. Therefore, $\mathcal{H}_S$ is a
hypergraph of minimum order. More generally, we say that a
$k-$regular, simple hypergraph $ \mathcal{H}$ with either $n$ or
$n-1$ edges has \emph{minimum order} if it does not exist a
$k-$regular, simple hypergraph $ \mathcal{H}'$ with either $n$ or
$n-1$ edges  and $ |V(\mathcal{H'})|<|V(\mathcal{H})|$. Thus, the
following lemma is straightforward.

\begin{lemma}\label{detmin}
A vertex set $S$ is a minimum determining set of $K_{n:k}$ with
$n\geq 2k+1$ if and only if the hypergraph $\mathcal{H}_S$ is
simple, has either $n$ or $n-1$ edges and  minimum order.
\end{lemma}

\begin{remark}
The characterization provided by Lemma \ref{detmin} ensures us
that it does not exist an infinite family of Kneser graphs
$K_{n:k}$ with fixed determining number, say $d$. Indeed, the
hypergraph associated to any minimum determining set must have a
fixed number of vertices $d$ which implies that neither $k$ nor
$n$ can take infinite values.
\end{remark}

Lemmas \ref{conditions1} and \ref{conditions} are the main tools
in order to compute the determining number of all Kneser graphs
$K_{n:k}$ with $n\geq \frac{k(k+1)}{2}+1$, what is done in
Theorems \ref{discretes} and \ref{gaps} below.

\begin{theorem} \label{discretes}
Let $k$ and $d$ be two positives integers such that $k\leq d$ and
$d>2$. Then, $${\rm Det}\left(K_{\left\lfloor
\frac{d(k+1)}{2}\right\rfloor+1:k}\right)=d$$
\end{theorem}

\begin{proof}
We first show that ${\rm Det}(K_{\lfloor
\frac{d(k+1)}{2}\rfloor+1:k})\leq d$. By Lemma \ref{conditions},
it suffices to prove that there exists a $k$-regular simple
hypergraph, say $\mathcal{H}_{k,d}$, with order $d$ and either
$\lfloor \frac{d(k+1)}{2}\rfloor$ or $\lfloor
\frac{d(k+1)}{2}\rfloor+1$ edges. We shall use the fact that
every complete graph $K_d$ has $d-1$ pairwise disjoint perfect
matchings whenever $d$ is even, and $\frac{d-1}{2}$ pairwise
disjoint hamiltonian cycles whenever $d$ is odd (see for
example~\cite{chromaticgraphtheory}). We distinguish three cases
according to the parity of $k$ and $d$.

\,

\emph{Case 1.} $d$ even: Consider $d$ vertices, a loop attached
at each vertex, and the edges of $k-1$ pairwise disjoint perfect
matchings of the complete graph $K_d$ (see Figure
\ref{matchingham}(a)). Clearly, this hypergraph
$\mathcal{H}_{k,d}$ is $k$-regular and has $\frac{d(k+1)}{2}$
edges. Note that its construction does not depend on the parity
of $k$.

\emph{Case 2.} $d$ odd and $k$ odd: $\mathcal{H}_{k,d}$ is the
hypergraph formed by $d$ vertices with loops attached at each
vertex, and $\frac{k-1}{2}$ pairwise disjoint hamiltonian cycles
of $K_d$ (see Figure \ref{matchingham}(b)). It is easy to check
that $\mathcal{H}_{k,d}$ is a $k$-regular hypergraph with
$\frac{d(k+1)}{2}$ edges.

\begin{figure}[ht]
\begin{center}

\begin{tabular}{cccc}
\includegraphics[width=0.2\textwidth]{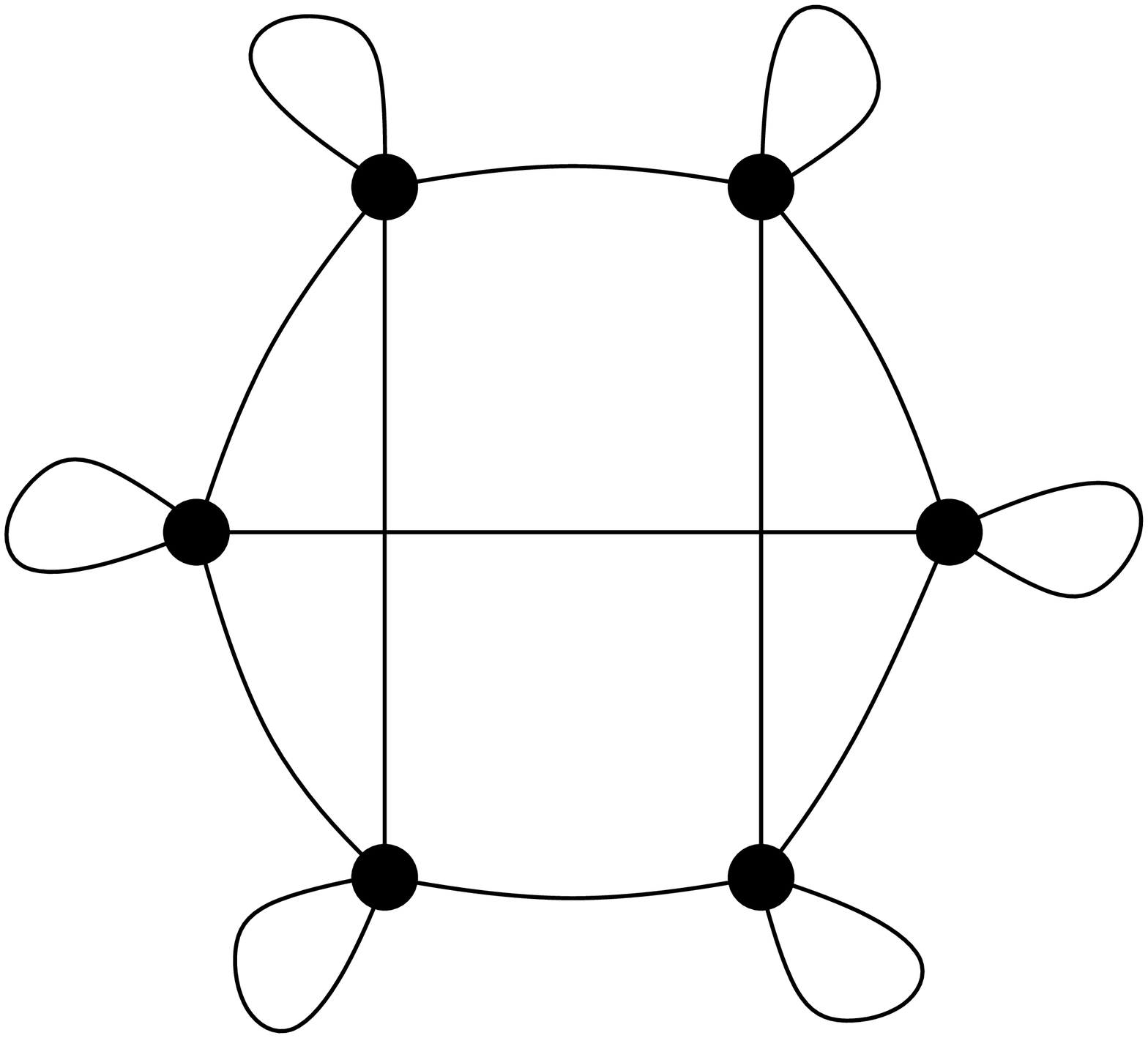}&
 & &
\includegraphics[width=0.2\textwidth]{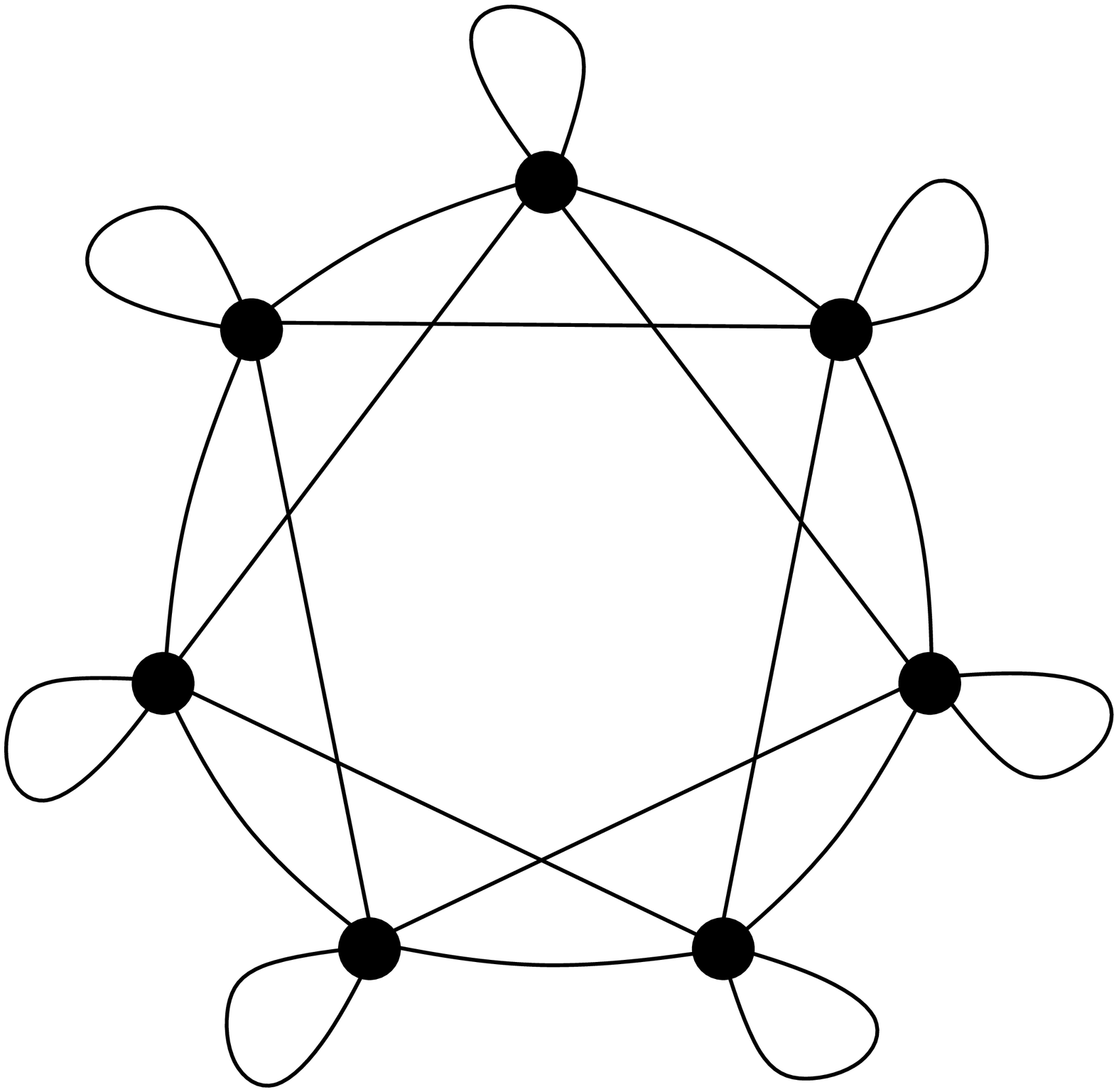}\\ & & \\
(a) & &  &  (b) \\

\end{tabular}
\end{center}
\caption{(a) Hypergraph $H_{4,6}$ constructed by considering $6$
vertices, a loop attached at each vertex, and the edges of $3$
pairwise disjoint perfect matchings of the complete graph $K_6$,
(b) Hypergraph $H_{5,7}$. It has $7$ vertices with loops attached
at each vertex, and $2$ pairwise disjoint hamiltonian cycles of
$K_7$. }\label{matchingham}
\end{figure}

\emph{Case 3.} $d$ odd and $k$ even: Consider the hypergraph
$\mathcal{H}_{k+1,d}$ obtained from case 2. Take any hamiltonian
cycle $C$ with edges, say $\{e_1, e_2, \ldots ,e_d\}$. Now,
delete the edges with even index and the loop attached at the
vertex in which $e_1$ and $e_d$ are incident. This construction
gives rise to a $k-$regular hypergraph $\mathcal{H}_{k,d}$ with
$\frac{d(k+1)-1}{2}=\lfloor \frac{d(k+1)}{2}\rfloor$ edges.

\

To complete the proof, it remains to show that ${\rm
Det}(K_{\lfloor \frac{d(k+1)}{2}\rfloor+1:k})$ is exactly equal
to $d$. By Lemma \ref{conditions}, it suffices to prove that
every $k$-regular hypergraph with either $\lfloor
\frac{d(k+1)}{2}\rfloor$ or $\lfloor \frac{d(k+1)}{2}\rfloor+1$
edges has at least $d$ vertices. Suppose on the contrary that
there exists a $k-$regular hypergraph $\mathcal{H}$ with $\lfloor
\frac{d(k+1)}{2}\rfloor$ edges (analogous for $\lfloor
\frac{d(k+1)}{2}\rfloor+1$ edges) and $d'<d$ vertices. By Theorem
2.8 of \cite{GGL} it follows that $d'=|V(\mathcal{H})|\geq \lceil
\frac{2|E(\mathcal{H})|}{k+1}\rceil$. Hence, $$d'\geq
\left\lceil\frac{2}{k+1}\left\lfloor\frac{d(k+1)}{2}\right\rfloor\right\rceil=d$$
which is a contradiction.

\end{proof}

Our next aim is to extend the result of Theorem \ref{discretes} to
Kneser graphs $K_{n:k}$ verifying that $d\geq k$, $d>2$ and
$\lfloor\frac{(d-1)(k+1)}{2}\rfloor < n-1
<\lfloor\frac{d(k+1)}{2}\rfloor$.  For our purpose, we first need
to introduce an operation on the edge set of any hypergraph
$\mathcal{H}$.

Let $e_1=\{v_1,v_2, \ldots, v_s\}$ and $e_2=\{w_1,w_2, \ldots,
w_r\}$ be two edges of $\mathcal{H}$ of sizes $s,r\geq 1$ with
possibly common vertices. We say that $e_1$ and $e_2$ \emph{are
merged} obtaining a new hypergraph $\mathcal{H'}$ if
$V(\mathcal{H'})=V(\mathcal{H})$ and
$E(\mathcal{H'})=(E(\mathcal{H})\setminus \{e_1,e_2\})\cup
\{v_1,\ldots, v_s,w_1, \ldots, w_r\}$. Denote by $e_1\cup e_2$
the set $\{v_1,\ldots, v_s,w_1, \ldots, w_r\}$ in which obviously
the possible common vertices are considered only once. Now, we
can extend this operation to \emph{merge a finite set of edges},
say $\{e_1, e_2, \ldots, e_t\}$ obtaining the hypergraph
$\mathcal{H'}$ with $E(\mathcal{H'})=(E(\mathcal{H})\setminus
\{e_1, e_2, \ldots, e_t\})\cup (e_1\cup e_2 \cup \ldots e_t)$.
Note that $|E(\mathcal{H'})|=|E(\mathcal{H})|-t+1$. We shall
apply the operation of \emph{merging edges} in regular
hypergraphs and for $e_i\neq e_j$ whenever $i\neq j$. Observe
that if $\mathcal{H}$ is $k-$regular and $e_1, e_2, \ldots, e_t$
are pairwise disjoint edges, that is, they have no vertex in
common, then $\mathcal{H'}$ is $k-$regular (see Figure
\ref{merging}).

\begin{figure}[ht]
\begin{center}

\begin{tabular}{cccc}
\includegraphics[width=0.2\textwidth]{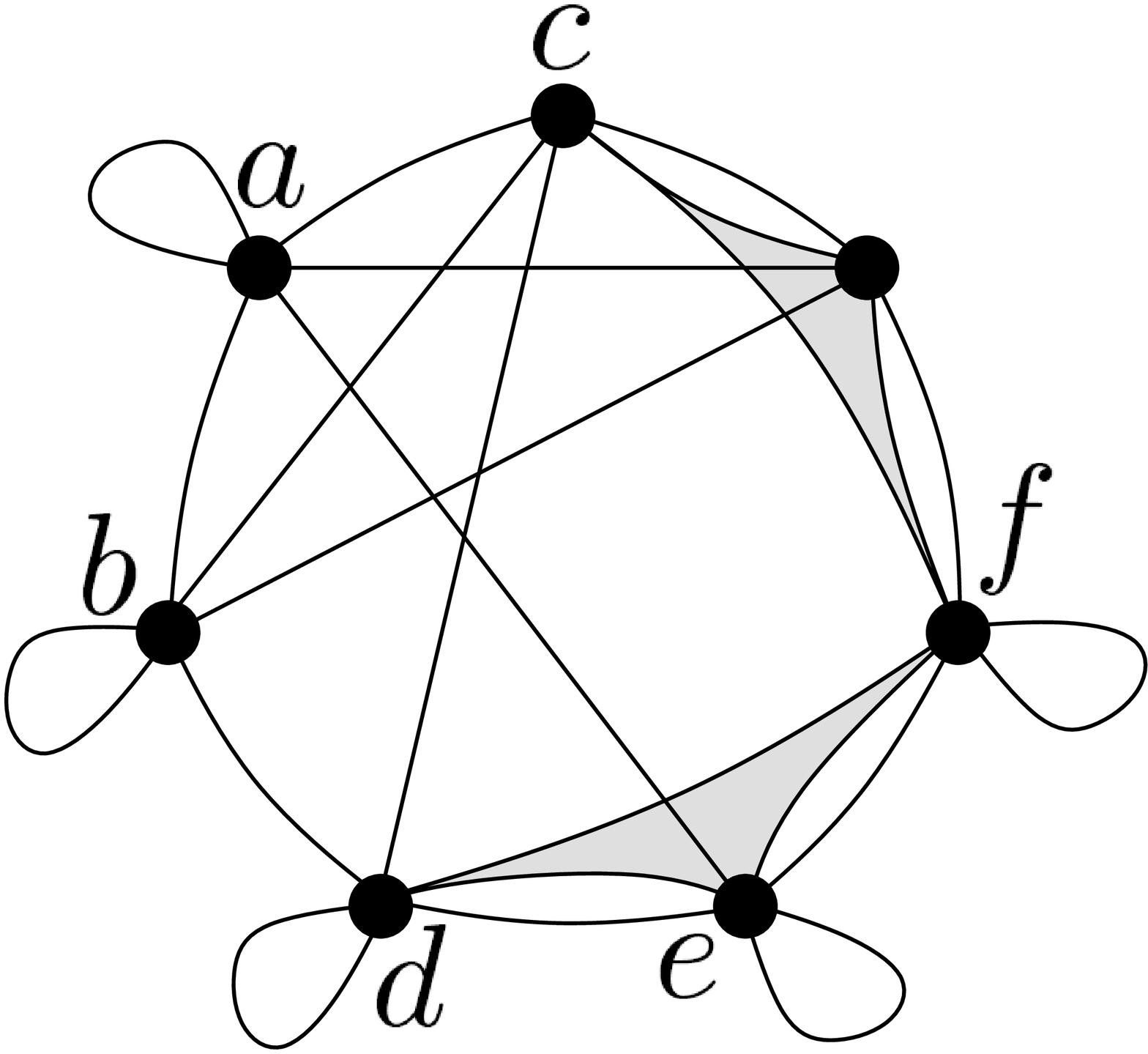}&
 & &
\includegraphics[width=0.2\textwidth]{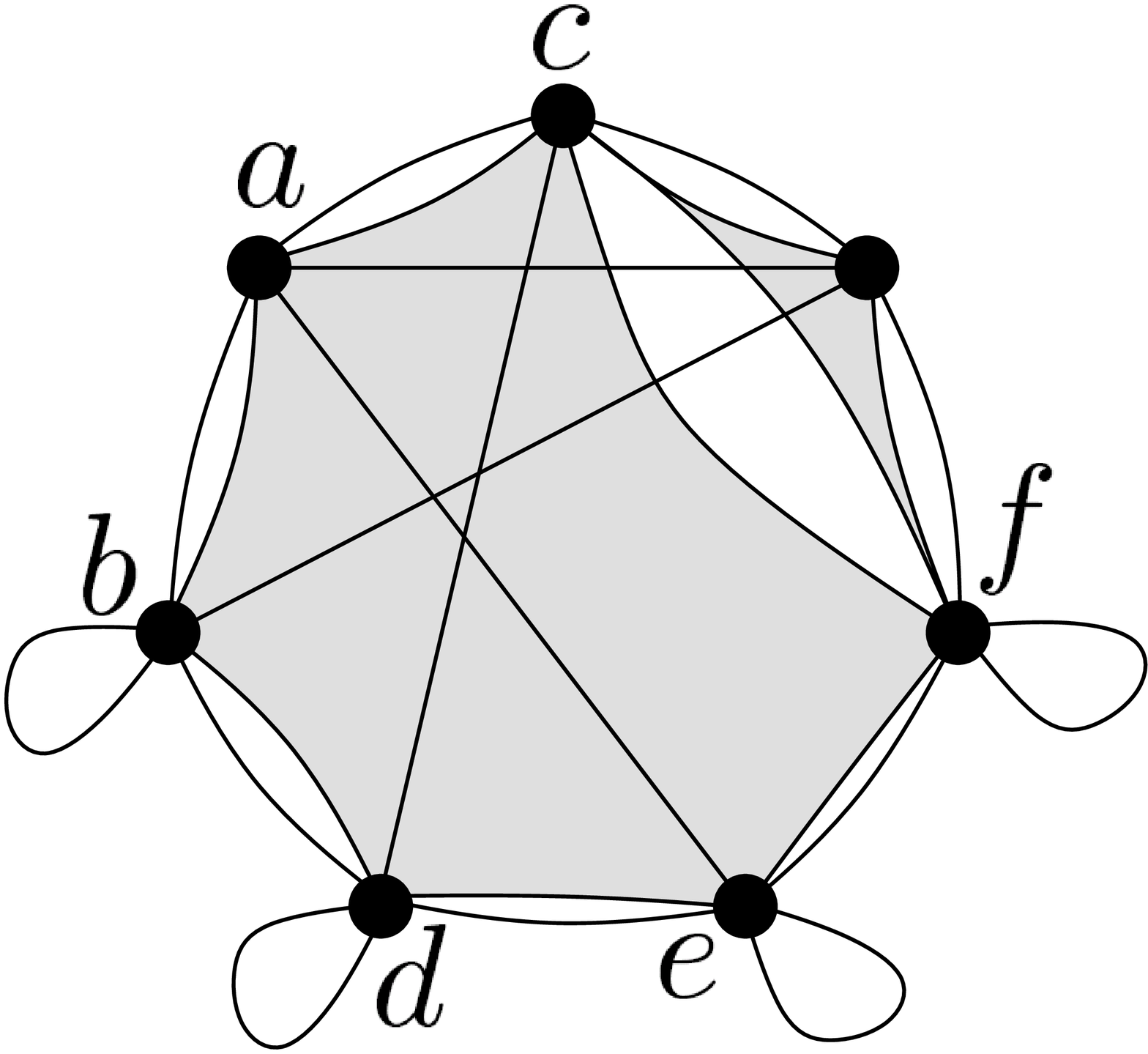}\\ & & \\
(a) & &  &  (b) \\

\end{tabular}
\end{center}
\caption{The edge $\{a,b,c,d,e,f\}$ of the hypergraph in (b) is
the result of merging three edges of the hypergraph in (a):
$\{a\}$, $\{b,c\}$ and $\{d,e,f\}$. Both hypergraphs are
$5-$regular.}\label{merging}
\end{figure}

\begin{theorem} \label{gaps}
Let $k$ and $d$ be two positives integers verifying that $3\leq
k+1\leq d$. For every $ n\in \mathbb{N}$ such that
$\lfloor\frac{(d-1)(k+1)}{2}\rfloor < n
<\lfloor\frac{d(k+1)}{2}\rfloor$ it holds $\Det(K_{n+1:k})=d$.

\end{theorem}

\begin{proof}
Since $\lfloor\frac{(d-1)(k+1)}{2}\rfloor < n
<\lfloor\frac{d(k+1)}{2}\rfloor$ then there exists $r\in
\mathbb{N}$ such that $n=\lfloor\frac{d(k+1)}{2}\rfloor-r$ with
$r\leq\lfloor \frac{k-1}{2}\rfloor$ whenever $d$ is odd or $d$ is
even and $k$ is odd, and $r\leq\frac{k}{2}$ whenever $k$ is even
and $d$ is even. We first prove that Det$(K_{n+1:k})\leq d$ by
distinguishing four cases according to the parity of $d$ and $k$.
By Lemma~\ref{conditions}, it suffices to show that there exists
a $k$-regular simple hypergraph with $d$ vertices and $n$ edges.

\,

\emph{Case 1.} $d$ even and $k$ even: Consider the hypergraph
$\mathcal{H}_{k,d}$ constructed as in case 1 of the proof of
Theorem \ref{discretes}. Since $k+1\leq d$ then $k\leq d-2$ and
so $r\leq \frac{k}{2}\leq \frac{d-2}{2}<\frac{d}{2}$. Hence we
can take $r+1\leq \frac{d}{2}$ edges of any perfect matching on
the vertices of $\mathcal{H}_{k,d}$ and merge them obtaining the
hypergraph $\mathcal{H'}_{k,d}$. Since the edges of the perfect
matching are disjoint then $\mathcal{H'}_{k,d}$ is $k-$regular.
Moreover, by construction
$d=|V(\mathcal{H}_{k,d})|=|V(\mathcal{H'}_{k,d})|$  and
$$|E(\mathcal{H'}_{k,d})|=|E(\mathcal{H}_{k,d})|-r=\frac{d(k+1)}{2}-r=n$$

\,

\emph{Case 2.} $d$ even and $k$ odd: Analogous to the previous
case but considering, $$r\leq \left\lfloor
\frac{k-1}{2}\right\rfloor \leq \left\lfloor
\frac{d-2}{2}\right\rfloor<\frac{d}{2}$$ and so $r+1\leq
\frac{d}{2}$.

\,

\emph{Case 3.} $d$ odd and $k$ odd: $\mathcal{H}_{k,d}$ is
constructed by considering $d$ vertices with loops attached at
each vertex, and $\frac{k-1}{2}$ pairwise disjoint hamiltonian
cycles (see case 2 of the proof of Theorem \ref{discretes}). Each
cycle has $d$ edges and $r\leq \frac{k-1}{2}\leq \frac{d-3}{2}$,
then we can merge $r$ disjoint edges of any hamiltonian cycle $C$.
Denote by $\{e_1,\ldots,e_d\}$ the edge set of $C$. Merge those
edges with even index plus the loop attached to the vertex in
which $e_1$ and $e_d$ are incident, $r+1$ disjoint edges in total.
Thus, we obtain a $k-$regular simple hypergraph
$\mathcal{H'}_{k,d}$ with $d$ vertices and
$n=\lfloor\frac{d(k+1)}{2}\rfloor-r$ edges.

\,

\emph{Case 4.} $d$ odd and $k$ even: $\mathcal{H}_{k,d}$ is
$k-$regular, has $d$ vertices and $\frac{d(k+1)-1}{2}$ edges (see
case 3 of the proof of Theorem \ref{discretes}). Note that we can
merge $r+1$ disjoint edges of a hamiltonian cycle $C$ of order $d$
since $r\leq \lfloor \frac{k-1}{2}\rfloor=\frac{k-2}{2}$ and so
$r+1\leq \frac{k}{2}\leq \frac{d-1}{2}<\lceil \frac{d}{2}\rceil$.
It suffices to consider $r+1$ pairwise disjoint edges among the
odd labeled edges of $C$. The resulting hypergraph
$\mathcal{H'}_{k,d}$  is simple, $k-$regular, has $d$ vertices and
$n=\lfloor\frac{d(k+1)}{2}\rfloor-r$ edges.

\

It remains to prove that Det$(K_{n+1:k})= d$. Suppose on the
contrary that Det$(K_{n+1:k})\leq d-1$, then by
 Lemma \ref{conditions1} there exists a $k$-regular simple hypergraph
$\mathcal{H}$ with $d-1$ vertices and either $n+1$ or $n$ edges.

Assume first that $\mathcal{H}$ has $n+1$ edges. The
size-edge-sequence $r_1\geq ... \geq r_{n+1}$, where $r_i$ is the
size of the edge $e_i$, satisfies (see \cite{GGL})
$$\displaystyle\sum_{i=1}^{n+1}r_i=\displaystyle\sum_{v\in
V(H)}\delta (v)=k(d-1).$$ Note that the number of loops in
$\mathcal{H}$ is at most $d-1$, so the other $n+1-(d-1)=n-d+2$
edges have size at least $2$. Hence, we obtain the following
inequalities about the sum on the edge sizes:

$$\left.\begin{array}{c}
 \displaystyle\sum_{i=n+1-(d-2)}^{n+1}r_i\geq d-1 \\
\displaystyle\sum_{i=1}^{n+1-(d-2)-1}r_i\geq 2(n-d+2)
\end{array}\right\}\Rightarrow k(d-1)=\displaystyle\sum_{i=1}^{n+1}r_i\geq d-1+2(n-d+2)= 2n-d +3$$

Therefore, $n\leq\frac{(d-1)(k+1)}{2} -1$ which is a contradiction
since $\lfloor\frac{(d-1)(k+1)}{2}\rfloor < n$.

Suppose now that $\mathcal{H}$ has $n$ edges. Then,

$$\left.\begin{array}{c}
 \displaystyle\sum_{i=n-(d-2)}^{n}r_i\geq d-1 \\
\displaystyle\sum_{i=1}^{n-(d-2)-1}r_i\geq 2(n-d+1)
\end{array}\right\}\Rightarrow k(d-1)=\displaystyle\sum_{i=1}^{n}r_i\geq d-1+2(n-d+1)= 2n-d +1$$

Hence, $n \leq \frac{(d-1)(k+1)}{2}$ which contradicts
$\lfloor\frac{(d-1)(k+1)}{2}\rfloor < n$.

\end{proof}

As it was said before, Theorems~\ref{discretes} and \ref{gaps}
provide the determining number for all Kneser graphs $K_{n:k}$
with $n\geq \frac{k(k+1)}{2}+1$. We want to stress the usefulness
of our technique by illustrating in Figure \ref{comparing}  the
values of $n$ and $k$ for which we have computed Det$(K_{n:k})$,
those obtained in \cite{boutin}, the trivial cases, and the
values of $n$ and $k$ for which Det$(K_{n:k})$ remains to compute.

\begin{figure}[ht]
\begin{center}
\includegraphics[width=65mm]{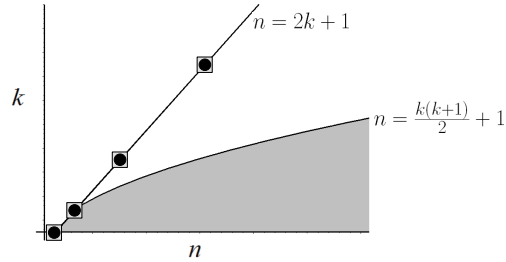} 
\caption{The shadow area corresponds to the values of $n$ and $k$
for which the determining number is provided by Theorems
\ref{discretes} and \ref{gaps}, the squared points are the values
obtained in \cite{boutin}. For $n<2k+1$ we have trivial cases.
The values that remain to compute are those on the line $n=2k+1$
with $n\neq 2^r-1$ and all the values in between the line
$n=2k+1$ and the curve $n=\frac{k(k+1)}{2}+1$. }\label{comparing}
\end{center}
\end{figure}

\section{Kneser graphs with fixed determining number}

In \cite{boutin}  Boutin characterizes all Kneser graphs with
determining numbers 2, 3 or 4, for which she has to assemble a
heavy machinery. Our technique allows us to prove the same results
but with shorter proofs. This is done in Proposition
\ref{given234} below. Moreover, the approach of determining sets
by hypergraphs lets us go further, obtaining all Kneser graphs
with determining number 5. We first need a technical lemma.

\begin{lemma} \label{cuentas}
Let $\mathcal{H}$ be a $k$-regular simple hypergraph with $d$
vertices and $m$ edges. Then the following statements hold:

\begin{enumerate}
\item[{\rm (a)}] \label{cuentas1} $k\leq 2^{d-1}$ and $m\leq 2^{d}-1$.
\item[{\rm (b)}]\label{cuentas2} If $m> d+\binom{d}{2}$ then $kd\geq 3m-2d-\binom{d}{2}$.
\item[{\rm (c)}] \label{cuentas3} If $m> d+\binom{d}{2}+\binom{d}{3}$ then $kd\geq 4m-3d-2\binom{d}{2}-\binom{d}{3}.$
\end{enumerate}

\end{lemma}

\begin{proof}

Statement (a) follows from the fact that the cardinality of the
power set $\mathcal{P}(V(\mathcal{H}))$ of the vertex set
$V(\mathcal{H})$ is $2^{d}$, and a hyperedge is a non-empty
subset of vertices.

To prove statement (b), assume that $m > d+\binom{d}{2}$ and
consider the size-edge sequence of $\mathcal{H}$, say  $r_1\geq
r_2\geq \ldots \geq r_{m-1}\geq r_m$. Then, $$\displaystyle
kd=\sum_{i=1}^{m}r_i = \sum_{i=m-d+1}^{m}r_i +
\sum_{i=m-\left(d+\binom{d}{2}\right)+1}^{m-d}r_i +
\sum_{i=1}^{m-\left(d+\binom{d}{2}\right)}r_i  \geq  $$ $$ \geq d
+ 2\binom{d}{2} + 3\left(m-d-\binom{d}{2}\right)=
3m-2d-\binom{d}{2}.$$

Suppose now that $m > d+\binom{d}{2}+\binom{d}{3}$. We have,
$$\displaystyle kd=\sum_{i=1}^{m}r_i = \sum_{i=m-d+1}^{m}r_i +
\sum_{i=m-\left(d+\binom{d}{2}\right)+1}^{m-d}r_i +
\sum_{i=m-\left(d+\binom{d}{2}+\binom{d}{3}\right)+1}^{m-\left(d+\binom{d}{2}\right)}r_i
\sum_{i=1}^{m-\left(d+\binom{d}{2}+\binom{d}{3}\right)}r_i \geq
$$ $$ \geq d + 2\binom{d}{2} + 3\binom{d}{3}
+4\left(m-d-\binom{d}{2}-\binom{d}{3}\right)=
4m-3d-2\binom{d}{2}-\binom{d}{3}.$$ Hence, statement (c) holds.

\end{proof}

The following result comprises Propositions 12, 13 and 14 of
\cite{boutin}. The statements are the same but we provide shorter
proofs by using hypergraphs.

\begin{proposition} \label{given234}

\begin{enumerate}
\item[{\rm (a)}] The only Kneser graph with determining number 2 is $K_{3:1}$.
\item[{\rm (b)}] The only Kneser graphs with determining number 3 are $K_{4:1},$ $K_{5:2}$ and $K_{7:3}$.
\item[{\rm (c)}] The only Kneser graphs with determining number 4 are $K_{5:1},$ $K_{6:2},$
$K_{7:2},$ $K_{8:3},$ $K_{9:3},$ $K_{9:4},$ $K_{10:4},$
$K_{11:4},$ $K_{11:5},$ $K_{12:5},$ $K_{13:6}$ and $K_{15:7}.$
\end{enumerate}
\end{proposition}

\begin{proof}
\begin{enumerate}

\item[{\rm (a)}] Consider the Kneser graph $K_{n+1:k}$ and assume that Det$(K_{n+1:k})=2$.
By Lemma \ref{detmin}, there exists a $k$-regular simple
hypergraph with minimum order 2 and either $n$ or $n+1$ edges.
There are only three simple, regular hypergraphs with two
vertices: a pair of loops, an edge of size 2 and an edge of size
2 with a loop attached at each vertex. However, only the first
one is an associated hypergraph to the non-trivial Kneser graph
$K_{3:1}$.

\item[{\rm (b)}] Since every Kneser graph $K_{n+1:1}$ is isomorphic to the complete graph
$K_{n+1}$, then only $K_{4:1}$ can have determining number 3.
Consider now $K_{n+1:k}$ with $k\geq 2$ and  suppose that
Det$(K_{n+1:k})=3$. By Lemma \ref{detmin}, there exists a
$k$-regular simple hypergraph with minimum order 3 and either $n$
or $n+1$ edges. Then, Lemma \ref{cuentas} implies that $n\leq 7$
and $k\leq 4$. Since  $n+1\geq 2k+1$, we obtain the following
candidates: $K_{5:2}$, $K_{6:2}$, $K_{7:2}$, $K_{8:2}$, $K_{7:3}$
and $K_{8:3}$. By Theorems \ref{discretes} and \ref{gaps} it is
easy to check that the only Kneser graphs with
$\Det(K_{n+1:k})=3$ are $K_{5:2}$ and $K_{7:3}$. For instance,
the graph $K_{6:2}$ has determining number 4 since $n=5$ and
$\lfloor\frac{3(d-1)}{2}\rfloor < n <\lfloor\frac{3d}{2}\rfloor$
for $d=4$.

\item[{\rm (c)}] Reasoning as in the previous case, the only Kneser graph $K_{n+1:1}$ with determining number 4 is
$K_{5:1}$. Take $K_{n+1:k}$ with $k\geq 2$ and  assume that
Det$(K_{n+1:k})=4$. By Lemma \ref{detmin}, there exists a
$k$-regular simple hypergraph with minimum order 4 and either $n$
or $n+1$ edges. Hence, Lemma \ref{cuentas} implies that $2k\leq
n\leq 15$ and $2\leq k\leq 8$. Thus, we get a list of 42 candidate
graphs. When $n\geq \frac{k(k+1)}{2}$ (it happens for 29 graphs
among the 42) we can apply Theorems \ref{discretes} and
\ref{gaps} obtaining that only $K_{6:2}, K_{7:2}, K_{8:3},
K_{9:3}$ and $K_{11:4}$ have determining number 4.

Suppose now that $2k\leq n<\frac{k(k+1)}{2}$. Then, $4\leq k\leq
8$. When $k=4$, the possible values of $n$ are 8 or 9.
Figure~\ref{hypergraphsgiven4} shows $4$-regular hypergraphs with
order 4 and having 8 and 9 edges respectively. By Lemma
\ref{conditions}, we have Det$(K_{9:4})\leq 4$ and
Det$(K_{10:4})\leq 4$. Obviously, Det$(K_{9:4})=$Det$(K_{10:4})=4$
since they are not in the above list of graphs with determining
number either 2 or 3. When $k=5$, then $10\leq n \leq 14$. Also,
Lemma~\ref{cuentas} gives $n \leq 11$.
Figure~\ref{hypergraphsgiven4} shows that
Det$(K_{11:5})=$Det$(K_{12:5})=4$. Similarly, when $k$ is 6 or 7,
the only Kneser graphs with determining number 4 are $K_{13:6}$
and $K_{15:7}$, and Figure~\ref{hypergraphsgiven4} shows their
associated hypergraphs. Finally, when $k=8$ there are not
suitable values of $n$, and that completes the proof.

\end{enumerate}
\end{proof}

\begin{figure}[ht]
\begin{center}

\begin{tabular}{ccc}
\includegraphics[width=0.2\textwidth]{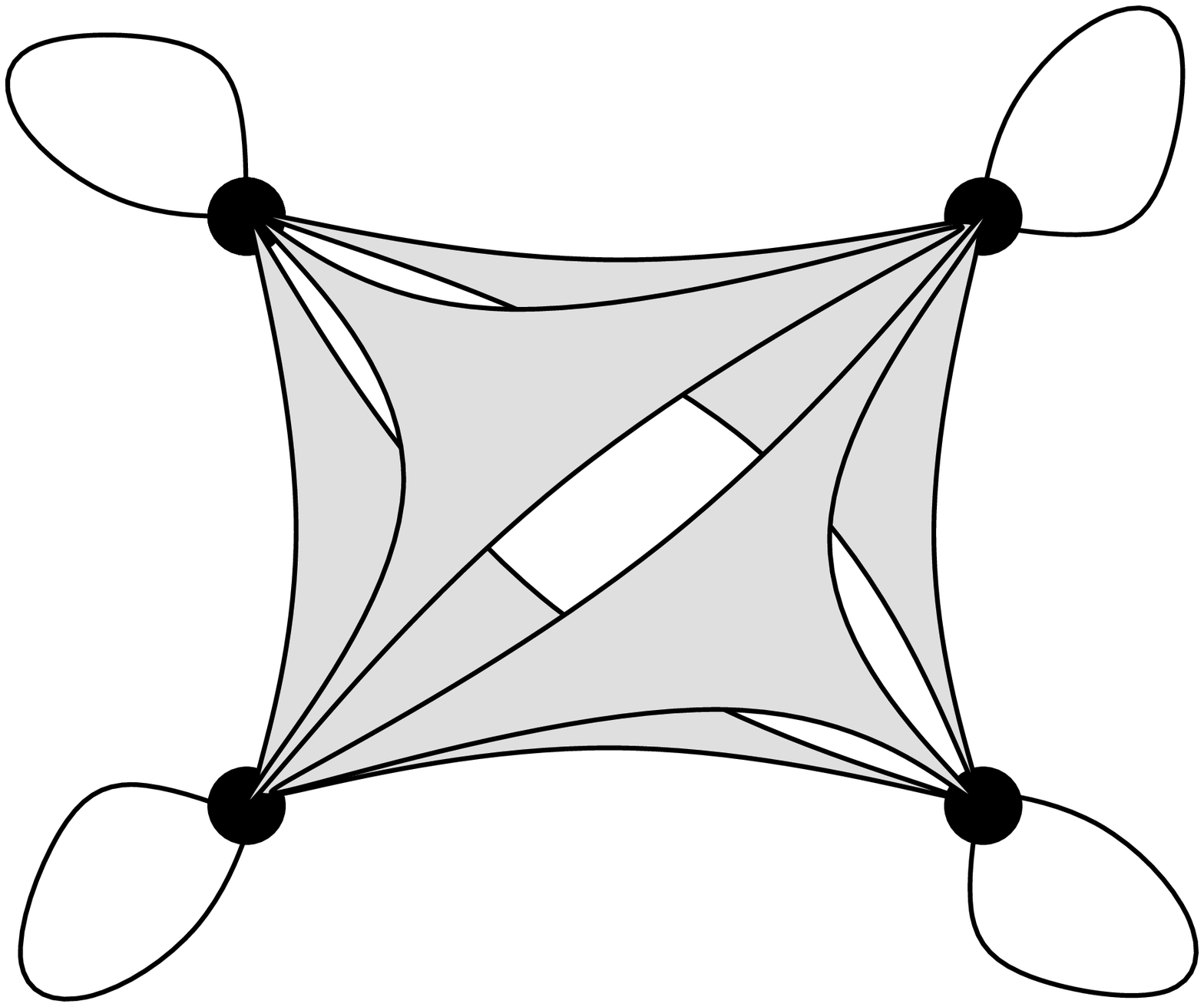}&
\includegraphics[width=0.2\textwidth]{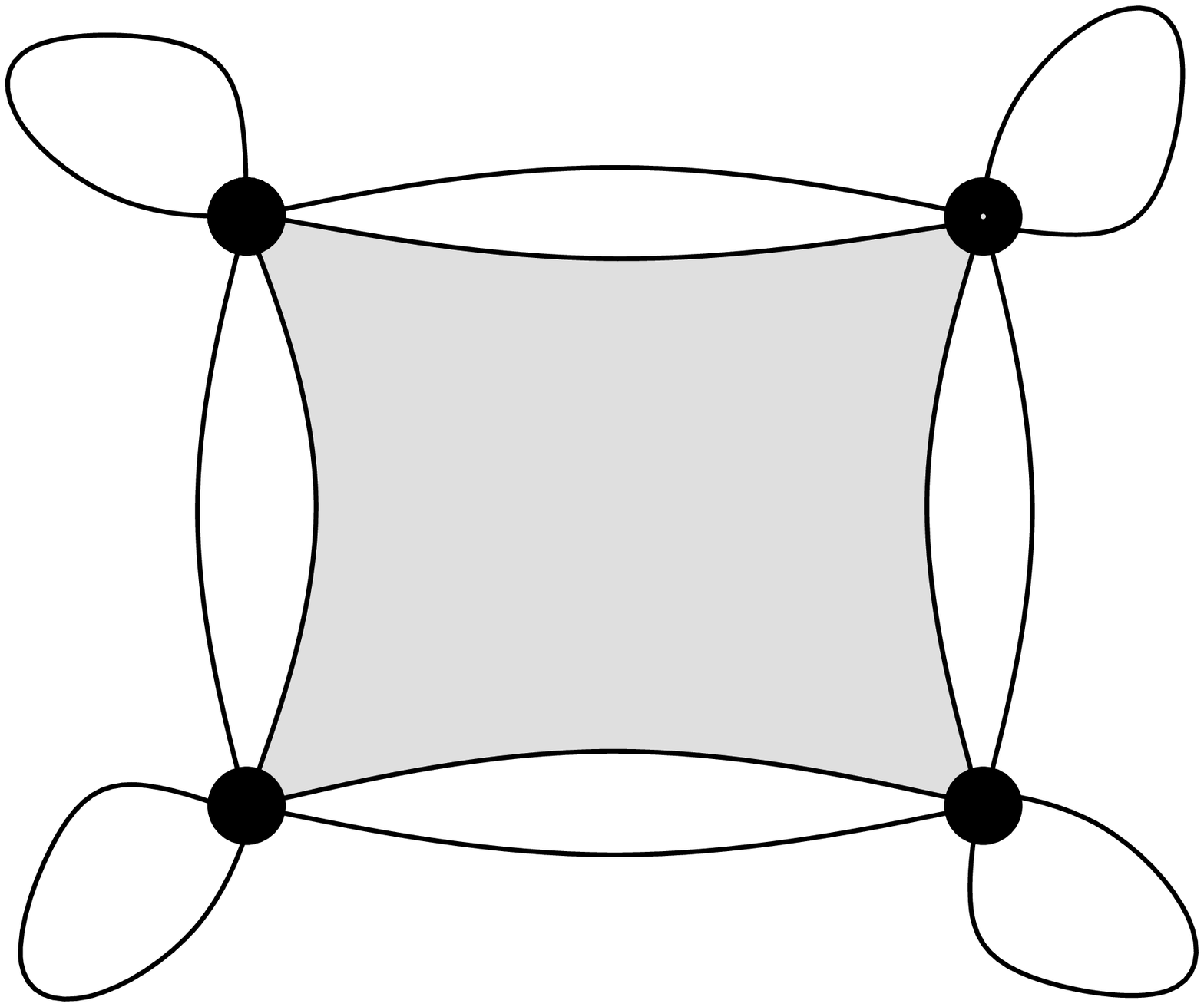} &
\includegraphics[width=0.2\textwidth]{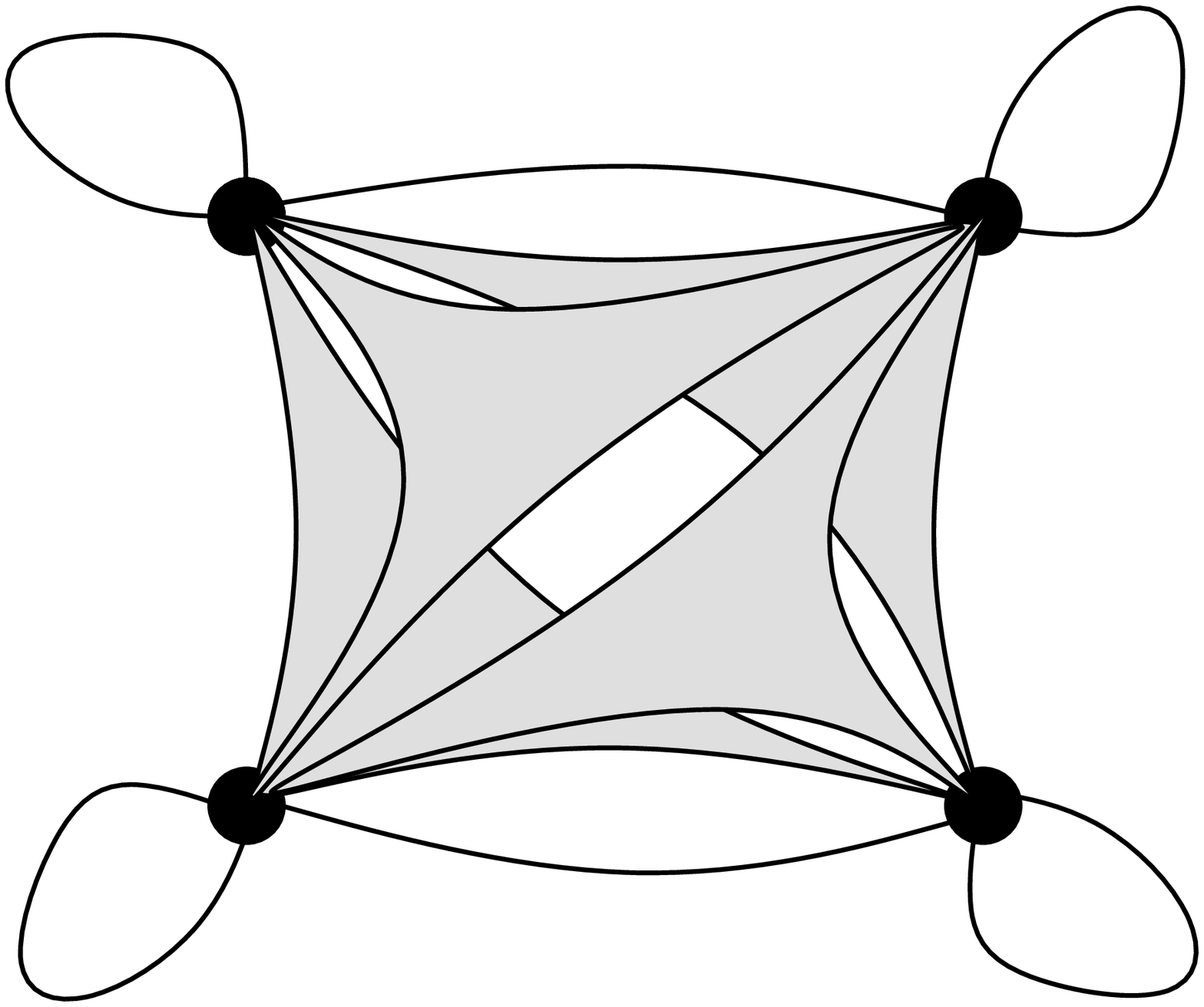}\\
$K_{9:4}$ & $K_{10:4}$ & $K_{11:5}$ \\
\includegraphics[width=0.2\textwidth]{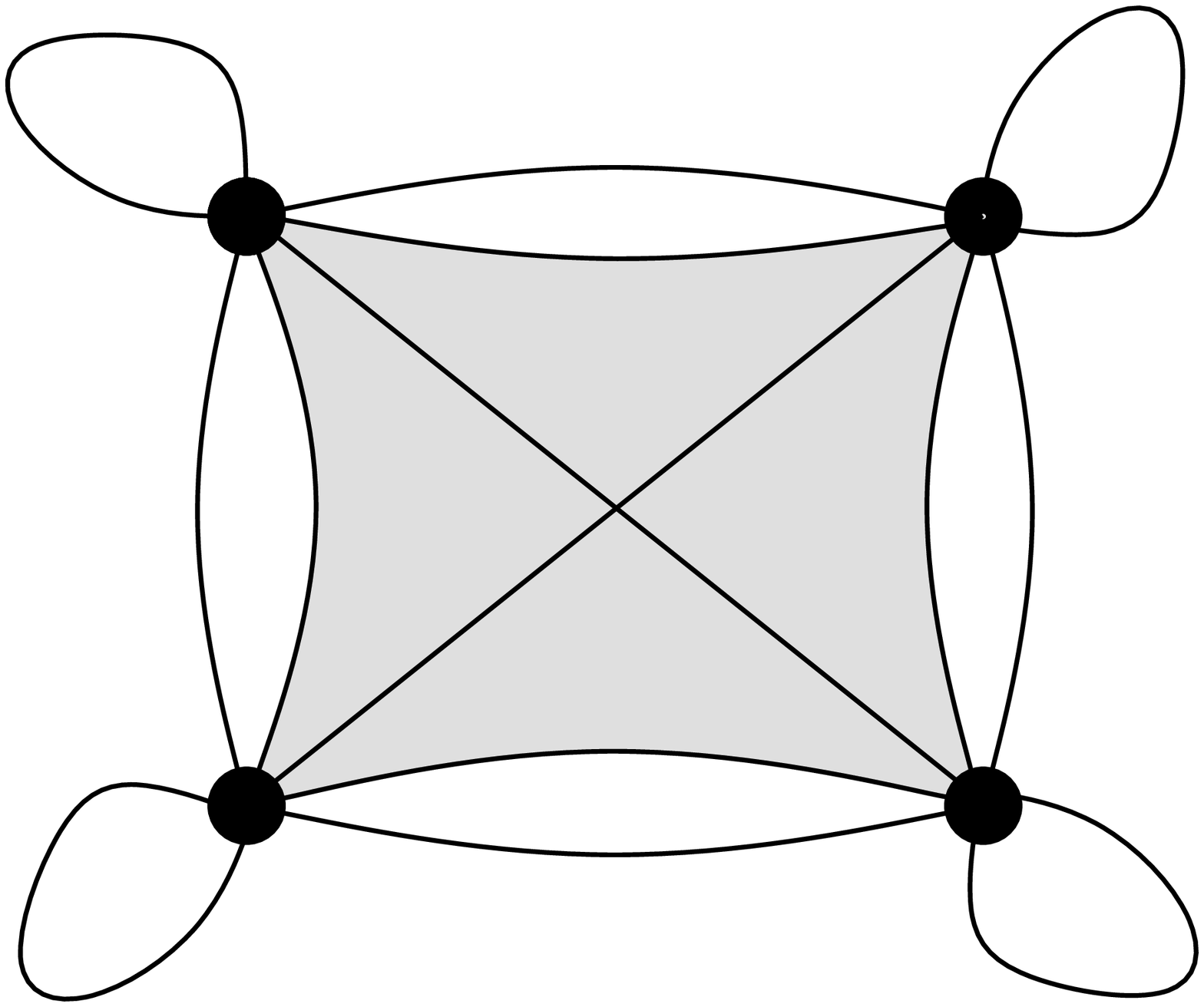}&
\includegraphics[width=0.2\textwidth]{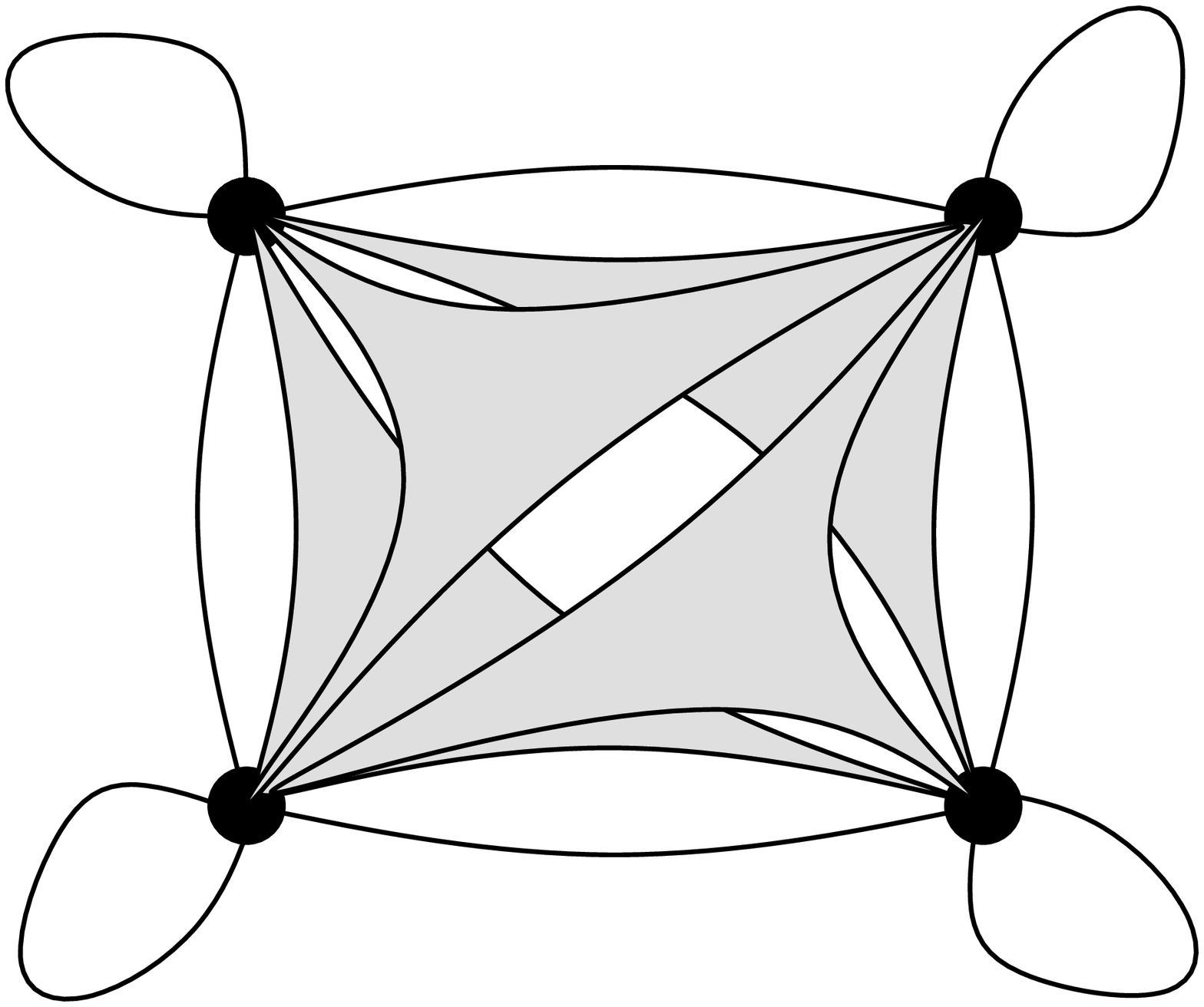}&
\includegraphics[width=0.2\textwidth]{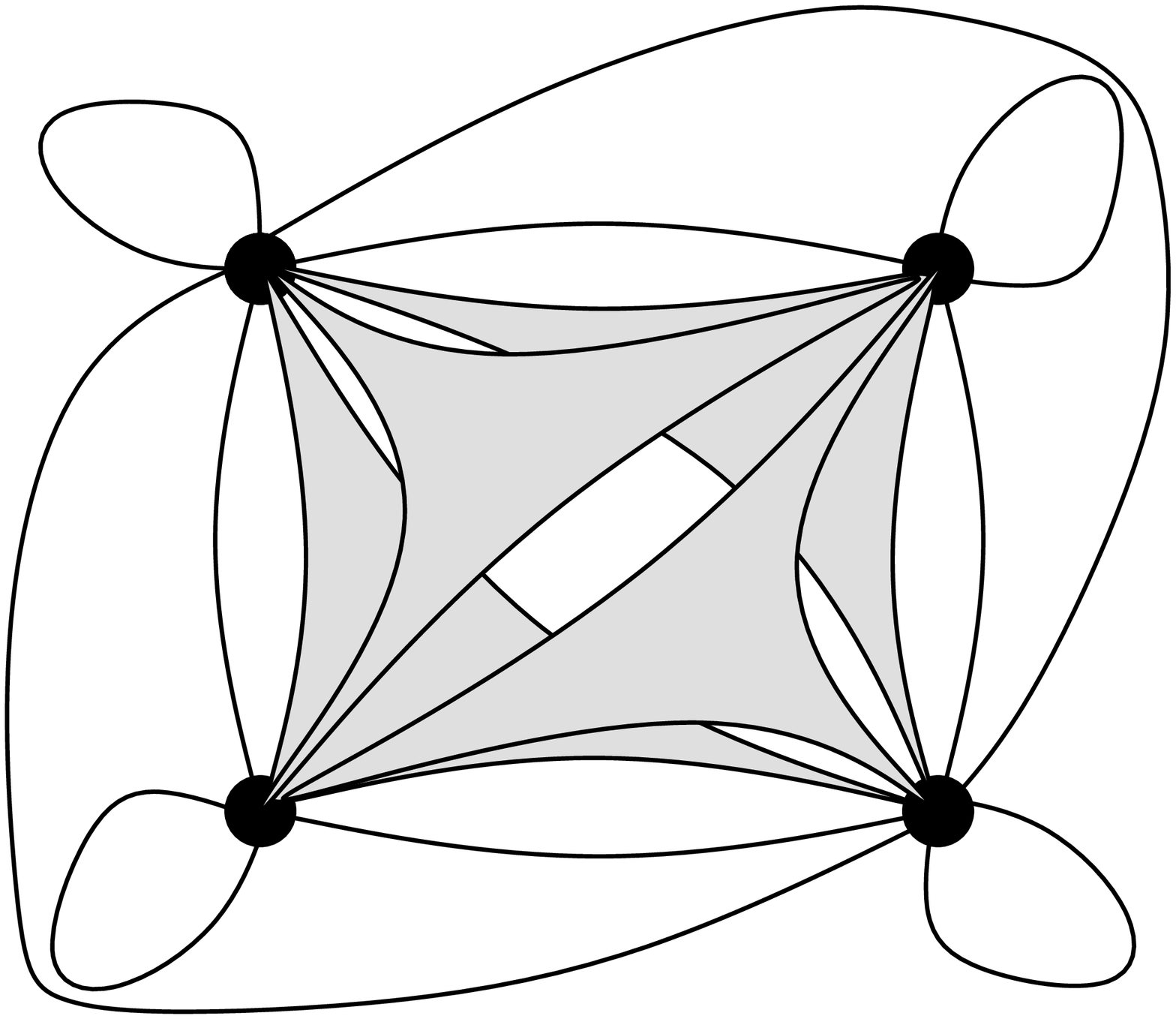}\\
$K_{12:5}$ & $K_{13:6}$ & $K_{15:7}$
\end{tabular}
\end{center}
\caption{Hypergraphs associated to Kneser graphs with determining
number 4.}\label{hypergraphsgiven4}
\end{figure}

Our technic lets us go further obtaining all Kneser graphs with
determining number 5.

\begin{proposition}\label{given5}
The Kneser graphs with determining number 5 are $K_{6:1},$
$K_{8:2},$ $K_{10:3},$ $K_{11:3}$, $K_{12:4},$ $K_{13:4}$,
$K_{13:5},$ $K_{14:5}$, $K_{15:5},$ $K_{16:5},$ $K_{14:6},$
$K_{15:6}$, $K_{16:6},$ $K_{17:6},$ $K_{16:7},$ $K_{17:7}$,
$K_{18:7},$ $K_{19:7},$ $K_{17:8},$ $K_{18:8}$, $K_{19:8},$
$K_{20:8},$ $K_{21:8},$ $K_{19:9},$ $K_{20:9}$, $K_{21:9},$
$K_{22:9},$ $K_{21:10},$ $K_{22:10}$, $K_{23:10},$ $K_{24:10},$
$K_{23:11},$ $K_{24:11}$, $K_{25:11},$ $K_{26:11},$ $K_{25:12},$
$K_{26:12}$, $K_{27:12},$ $K_{27:13},$ $K_{28:13}$, $K_{29:14},$
 and $K_{31:15}.$
\end{proposition}

\begin{proof}
As it was said before, the Kneser graph $K_{n+1:1}$ is isomorphic
to the complete graph $K_{n+1}$ then only $K_{6:1}$ has
determining number 5. Consider the graph $K_{n+1:k}$ with $k\geq
2$ and suppose that Det$(K_{n+1:k})=5$. By Lemma \ref{detmin},
there exists a $k$-regular simple hypergraph $ \mathcal{H}$ with
minimum order 5 and either $n$ or $n+1$ edges. By Lemma
\ref{cuentas} it follows that $2k\leq n\leq 31$ and $2\leq k\leq
16$. Thus, the list of candidate graphs increases now to 196
Kneser graphs. When $n\geq \frac{k(k+1)}{2}$ (it happens for 157
graphs among the 196) we can apply Theorems \ref{discretes} and
\ref{gaps} obtaining that only $K_{8:2}$, $K_{10:3}$, $K_{11:3}$,
$K_{12:4}$, $K_{13:4}$ and $K_{16:5}$ have determining number 5.

Assume now that $2k\leq n<\frac{k(k+1)}{2}$ what implies that
$4\leq k\leq 16$. When $k=4$, the possible values of $n$ are 8 or
9 but they correspond to Kneser graphs with determining number 4
(see Proposition~\ref{given234}). When $k=5$, Lemma~\ref{cuentas}
gives $10\leq n \leq 14$. However for $n$ equal to either 10 or
11, we obtain Kneser graphs already studied in
Proposition~\ref{given234}, whose determining numbers are equal
to 4. The remaining values correspond to the Kneser graphs
$K_{13:5}$, $K_{14:5}$ and $K_{15:5}$ whose associated
hypergraphs with 5 vertices are illustrated in
Figure~\ref{hypergraphsgiven5}. Table~\ref{table} shows the rest
of the values of $n$ and $k$ for which Det$(K_{n+1:k})=5$. In all
the cases, it is easy to construct the associated hypergraph. On
the other hand, note that for $k=6$ and $n=12$ or  $k=7$ and
$n=14$, the corresponding Kneser graph has determining number 4.
For the remaining available values of $n$ and $k$, we use
conditions (b) and (c) of Lemma \ref{cuentas} in order to show
that the hypergraph $ \mathcal{H}$ does not exist and hence the
determining in those cases cannot be equal to 5. For instance, if
$k=6$ and $n=18$ then it is straightforward to check that
condition (b) of Lemma \ref{cuentas} does not hold when $d=5$.

\begin{figure}[ht]
\begin{center}

\begin{tabular}{ccc}
\includegraphics[width=0.2\textwidth]{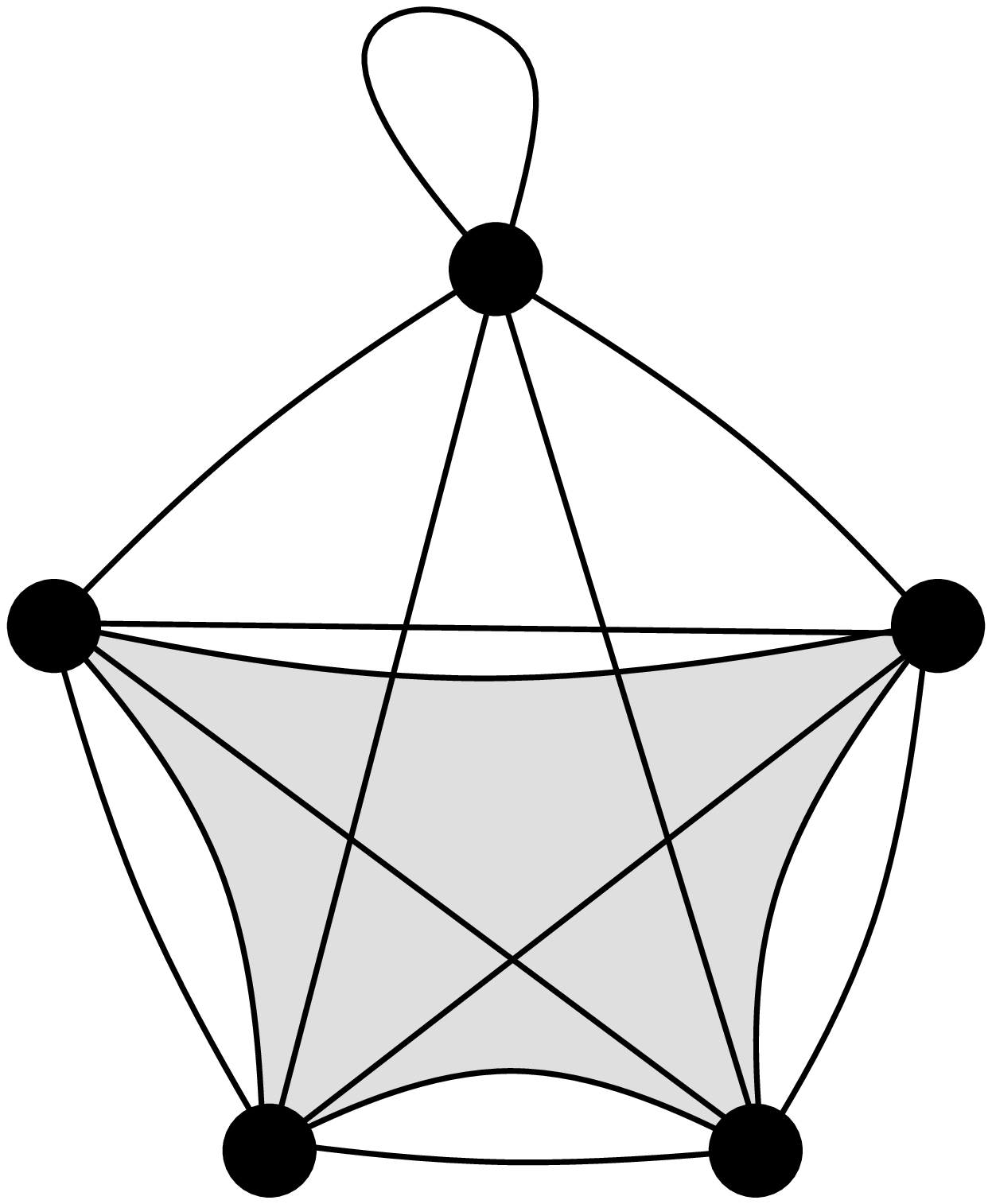}&
\includegraphics[width=0.2\textwidth]{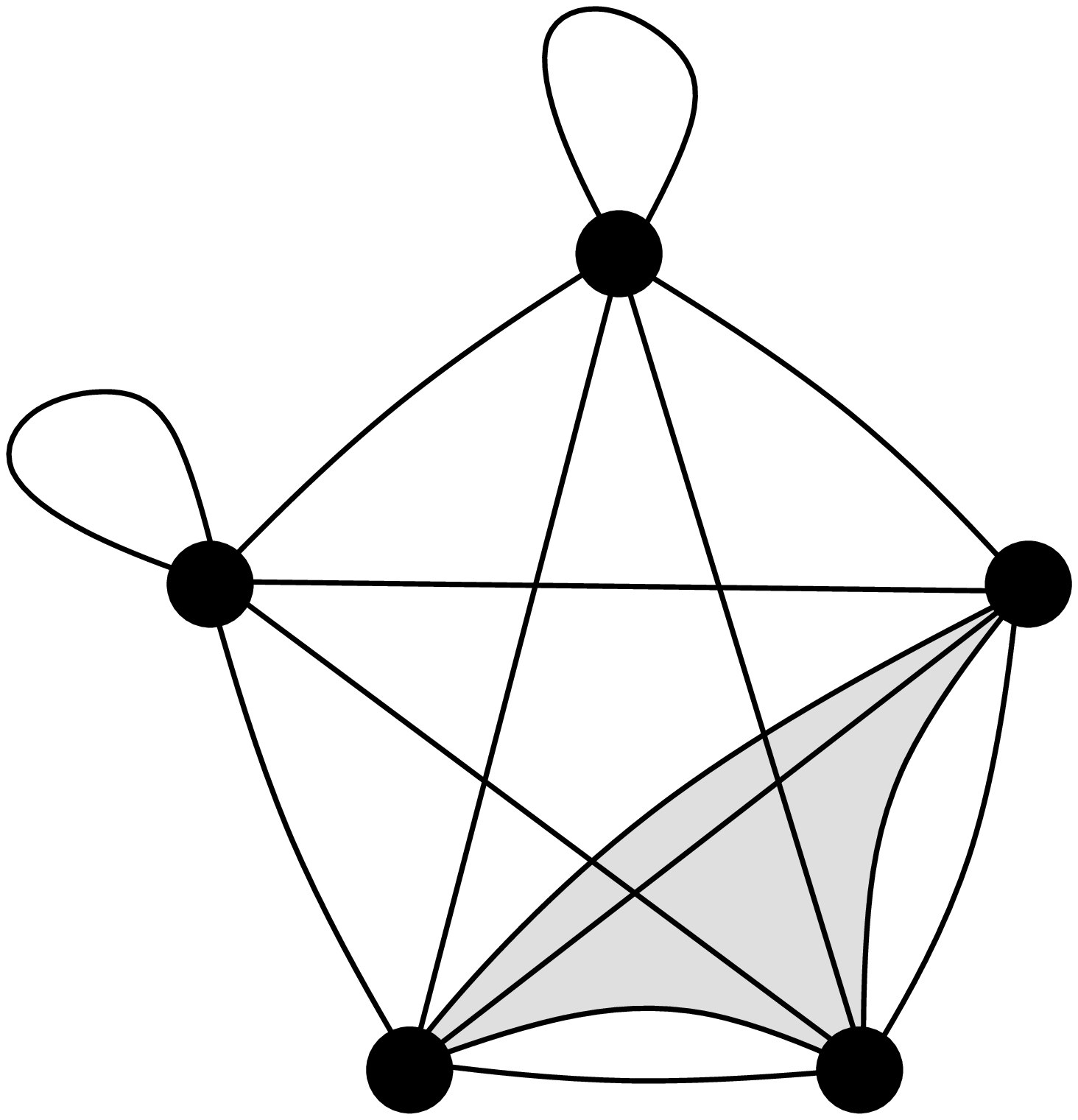} &
\includegraphics[width=0.2\textwidth]{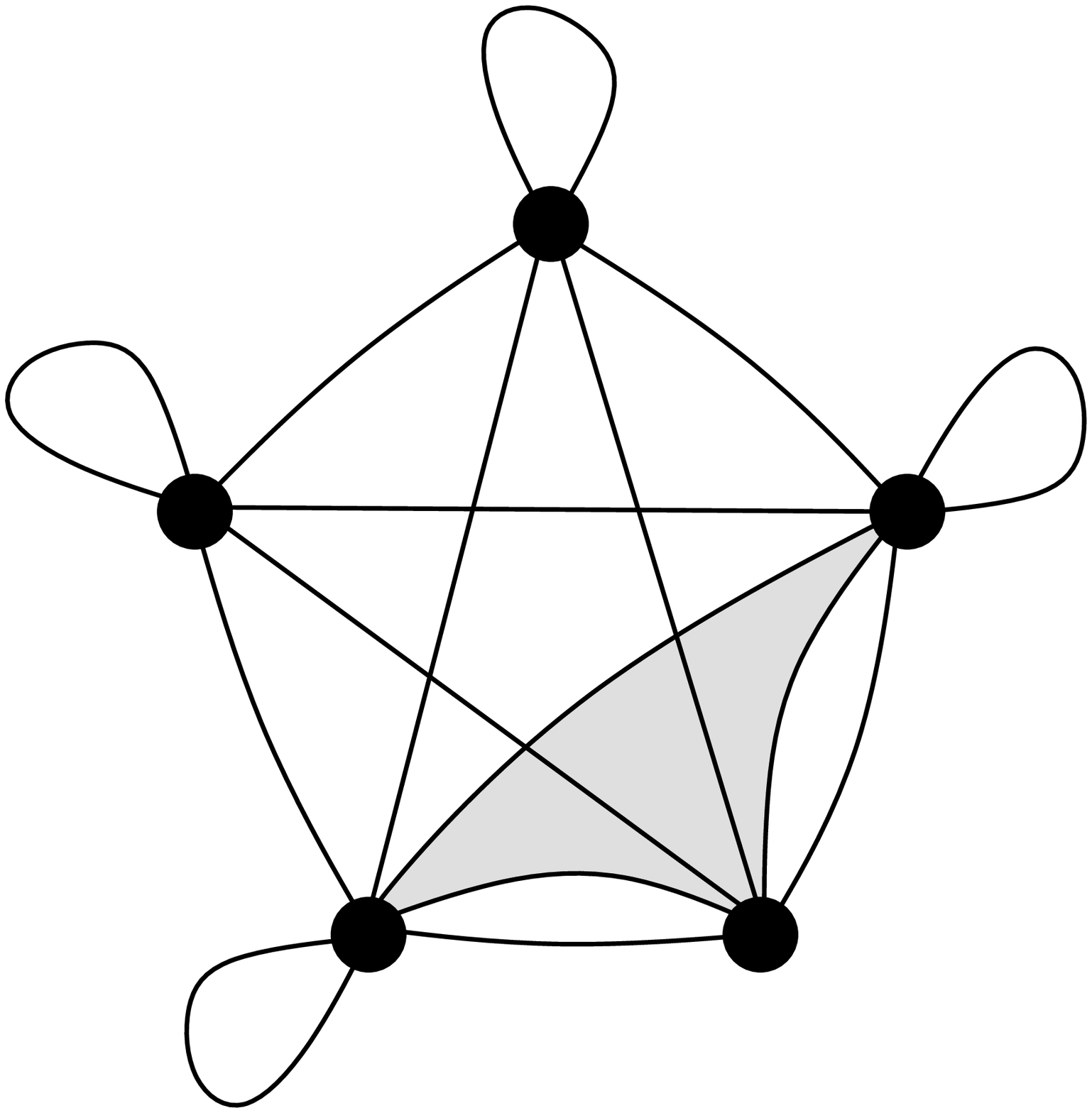}\\
$K_{13:5}$ & $K_{14:5}$ & $K_{15:5}$ \\

\end{tabular}
\end{center}
\caption{Hypergraphs associated to Kneser graphs with determining
number 5.}\label{hypergraphsgiven5}
\end{figure}

\begin{table}
\begin{center}
\begin{tabular}[center]{|c|c|} \hline

\textbf{$k$} & \textbf{$n$}  \\
\hline
6       & 13,14,15,16\\
\hline
7       & 15,16,17,18\\
\hline
8       & 16,17,18,19,20\\
\hline
9       & 18,19,20,21\\
\hline
10      & 20,21,22,23\\
\hline
11      & 22,23,24,25\\
\hline
12      & 24,25,26\\
\hline
13      & 26,27\\
\hline
14      & 28\\
\hline
15      & 30\\
\hline
\end{tabular}\caption{Values of $n$ and $k$ with $k\geq 6$ for which Det$(K_{n+1:k})=5$.} \label{table}
\end{center}
\end{table}

\end{proof}

\section{Kneser graphs $K_{n:k}$ with determining number $n-k$}

The characterization of determining sets in terms of hypergraphs
provided in Section 3 drives us to answer the question posed by
Boutin in \cite{boutin}: \emph{We know that Det$(K_{n:k})=n-k$ for
$K_{n:1}$ for any $n$, $K_{5:2}$ and $K_{6:2}$. Is there an
infinite family of Kneser graphs with $k\geq 2$ for which
Det$(K_{n:k})=n-k$?}

\begin{lemma}\label{nk}
Let $k$ and $n$ be two positive integers such that $2k\leq
n<\frac{k(k+1)}{2}$. Then, {\rm Det}$(K_{n+1:k})\leq k$.
\end{lemma}

\begin{proof}
By Lemma~\ref{conditions}, it suffices to prove that there exists
a $k$-regular simple hypergraph $\mathcal{H}$ with $k$ vertices
and $n$ edges. Consider the $k-$regular simple hypergraph
$\mathcal{H}_{k,d}$ constructed in the proof of
Theorem~\ref{discretes}. Recall that, independently of the parity
of $k$ and $d$, the hypergraph $\mathcal{H}_{k,d}$ is
$k-$regular, has $d$ vertices and $\lfloor \frac{d(k+1)}{2}
\rfloor$ edges. Assume that $k=d$, then we have a hypergraph
$\mathcal{H}_{k,k}$ with $k$ vertices and $\frac{k(k+1)}{2}$
edges.

Let $\{v_0,v_1,...,v_{k-1}\}$ be the vertex set of
$\mathcal{H}_{k,k}$. We distinguish two cases according to the
parity of $k$. Note that all the indices below are taken modulo
$k$.

\,

\emph{Case 1.} $k$ odd: $\mathcal{H}_{k,k}$ is the hypergraph
formed by $k$ vertices with loops attached at each vertex, and
$\frac{k-1}{2}$ pairwise disjoint hamiltonian cycles of $K_k$ (see
case 2 of Theorem \ref{discretes}). We assign the following set
of $\frac{k-1}{2}$ edges to each vertex $v_i\in
V(\mathcal{H}_{k,k})$ (see Figure \ref{finallemma1}):
$$\mathcal{E}_i=\{\{v_{i-1},v_{i+1}\},\{v_{i-2},v_{i+2}\},...,\{v_{i-\frac{k-1}{2}},v_{i+\frac{k-1}{2}}\}\}$$
Note that the edges of $\mathcal{E}_i$ are disjoint and two sets
$\mathcal{E}_i$, $\mathcal{E}_j$ have no edges in common whenever
$i\neq j$. These facts guarantee that in the process of merging
that we are going to describe next, the $k-$regularity is
preserved. We consider again two cases.

\begin{enumerate}

\item[1.1.] If $\frac{k(k+1)}{2}-\frac{k-1}{2}+1<
n<\frac{k(k+1)}{2}$ then merge a subset of $\frac{k(k+1)}{2}-n+1$
edges of $\mathcal{E}_0$, obtaining a $k-$regular hypergraph
$\mathcal{H}'_{k,k}$ with
$\frac{k(k+1)}{2}-(\frac{k(k+1)}{2}-n+1)+1=n$ edges.

\,

\item[1.2.] If $2k\leq n\leq\frac{k(k+1)}{2}-\frac{k-1}{2}+1$ then we can merge
the edges of at least one set $\mathcal{E}_i$ obtaining a
$k-$regular hypergraph which has a number of edges
 bigger or equal to $n$. If the number is $n$
then the process is concluded. Otherwise, suppose that we can
merge the edges of $s$ subsets with $0\leq s\leq k-1$, say
$\mathcal{E}_0, \mathcal{E}_2, \ldots ,\mathcal{E}_{s-1}$,
obtaining a $k-$regular hypergraph $\mathcal{H}'_{k,k}$ with
$\frac{k(k+1)}{2}-\frac{s(k-1)}{2}+s$ edges and verifying that
$$\frac{k(k+1)}{2}-\frac{(s-1)(k-1)}{2}+(s-1)<n<\frac{k(k+1)}{2}-\frac{s(k-1)}{2}+s$$
Then the edges of $\mathcal{E}_{s}$ cannot be merged since if so
the resulting hypergraph would have a number of edges smaller than
$n$. Hence, we proceed as in case 1.1 merging
$\frac{k(k+1)}{2}-\frac{s(k-1)}{2}+s-n+1$ edges of
$\mathcal{E}_s$. This process leads to a $k$-regular simple
hypergraph $\mathcal{H}$ with $k$ vertices and $n$ edges.

\end{enumerate}

\begin{figure}[ht]
\begin{center}

\begin{tabular}{ccc}
\includegraphics[width=0.25\textwidth]{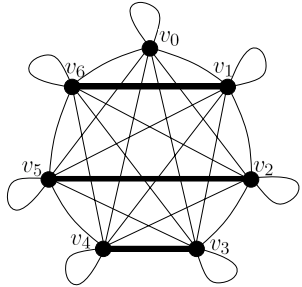}
\end{tabular}
\end{center}
\caption{Hypergraph $\mathcal{H}_{7,7}$. The selected edges form
the set $\mathcal{E}_{0}$.}\label{finallemma1}
\end{figure}

\emph{Case 2.} $k$ even: $\mathcal{H}_{k,k}$ is a hypergraph with
$k$ vertices, a loop attached at each vertex, and the edges of
$k-1$ pairwise disjoint perfect matchings of the complete graph
$K_k$ (see case 1 of Theorem \ref{discretes}). We distinguish
three cases.

\begin{enumerate}

\item[2.1.] If $\frac{k(k+1)}{2}-\frac{k}{2}(\frac{k}{2}-1)+\frac{k}{2}
\leq n<\frac{k(k+1)}{2}$ then we can follow an analogous process
of merging than in case 1 preserving also the $k-$regularity, but
instead of assigning the sets $\mathcal{E}_i$ to each vertex
$v_i$, we now assign the following set of edges to $v_i$ for
$i=0, \ldots, \frac{k}{2}-1$ (see Figure \ref{finallemma2}(a)):
$$\mathcal{F}_i=\{\{v_{i-1},v_{i+1}\},\{v_{i-2},v_{i+2}\},...,\{v_{i-\frac{k}{2}+1},v_{i+\frac{k}{2}-1}\}\}$$

Note first that the assignment is done to half of the vertices
since $\mathcal{F}_i=\mathcal{F}_{i+\frac{k}{2}}$ because of the
parity of $k$. Observe also that the edges of $\mathcal{F}_i$ are
disjoint and two sets $\mathcal{F}_i$, $\mathcal{F}_j$ have no
edges in common whenever $i\neq j$.

The process described in case 1 leads to a $k$-regular simple
hypergraph $\mathcal{H'}_{k,k}$  which is the result of merging
at most the edges of all sets $\mathcal{F}_i$ in
$\mathcal{H}_{k,k}$, that is, merging at most $
\frac{k}{2}(\frac{k}{2}-1)$ edges and obtaining in such case a
hypergraph with
$\frac{k(k+1)}{2}-\frac{k}{2}(\frac{k}{2}-1)+\frac{k}{2}$ edges.
Note that the edges obtained by this procedure are all of size
$k-2$ but at most one of smaller size.

\item[2.2.]
If
$\frac{k(k+1)}{2}-\frac{k}{2}(\frac{k}{2}-1)+\frac{k}{2}-\frac{k}{2}(\frac{k}{2}-2)+\frac{k}{2}=3k\leq
n<\frac{k(k+1)}{2}-\frac{k}{2}(\frac{k-2}{2}-1)$ then we first
merge all sets of edges $\mathcal{F}_i$, obtaining a hypergraph
$\mathcal{H}'_{k,k}$ with
$\frac{k(k+1)}{2}-\frac{k}{2}(\frac{k}{2}-1)+\frac{k}{2}$ edges.
We now assign the following set of edges to $v_i$ for $i=0,
\ldots, \frac{k}{2}-1$ (see Figure \ref{finallemma2}(b)):
$$\mathcal{F}'_i=\{\{v_{i-1},v_{i+2}\},\{v_{i-2},v_{i+3}\},...,\{v_{i-\frac{k}{2}+2},v_{i+\frac{k}{2}-1}\}\}$$

Again, we follow the procedure described in case 1 which gives a
$k$-regular simple hypergraph  that is the result of merging at
most the edges of all sets $\mathcal{F}'_i$ in
$\mathcal{H}'_{k,k}$, that is, merging at most $
\frac{k}{2}(\frac{k}{2}-2)$ edges and obtaining in such case a
hypergraph with
$\frac{k(k+1)}{2}-\frac{k}{2}(\frac{k}{2}-1)+\frac{k}{2}-\frac{k}{2}(\frac{k}{2}-2)+\frac{k}{2}$
edges. Observe that the edges obtained by this process are all of
size $k-4$ but at most one of smaller size.

\item[ 2.3.] If $2k\leq n < 3k$ then merge the sets $\mathcal{F}_i$
and  $\mathcal{F'}_i$, obtaining a hypergraph with $3k$ edges.
These edges are: $k$ loops, $\frac{k}{2}$ edges of size $k-2$,
$\frac{k}{2}$ edges of size $k-4$ and $k$ edges of size 2 forming
the cycle  $\{v_0,...v_{k-1}\}$. For every vertex $v_i$, consider
now the set of edges (see Figure \ref{finallemma2}(c)):
$$\mathcal{F}''_i=\{\{v_i\},\{v_{i+1},v_{i+2}\}\}$$ and merge the
required sets $\mathcal{F}''_i$ to attain a hypergraph with $n$
edges.

\end{enumerate}

\begin{figure}[ht]
\begin{center}

\begin{tabular}{ccc}
\includegraphics[width=0.25\textwidth]{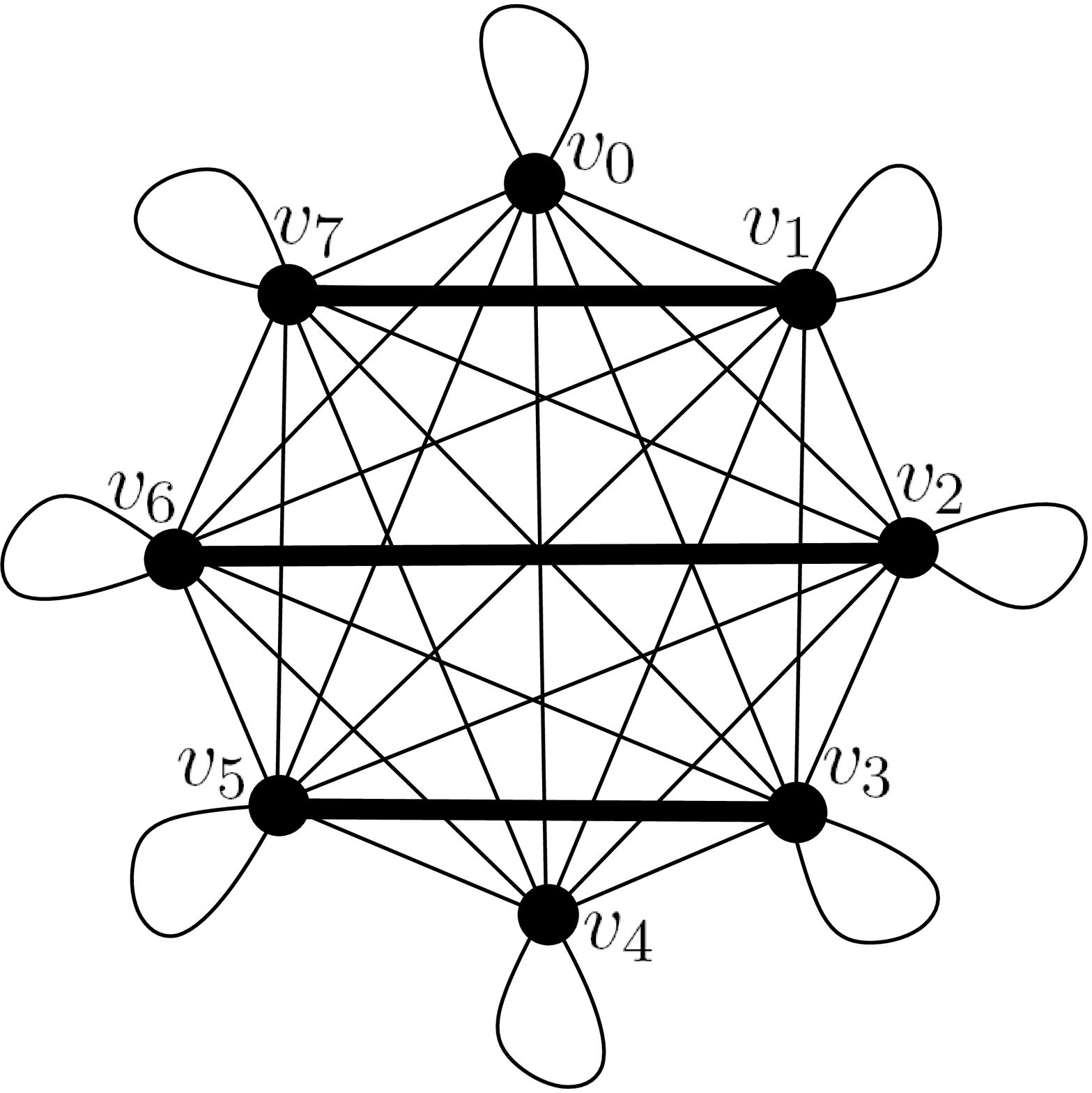}&
\includegraphics[width=0.25\textwidth]{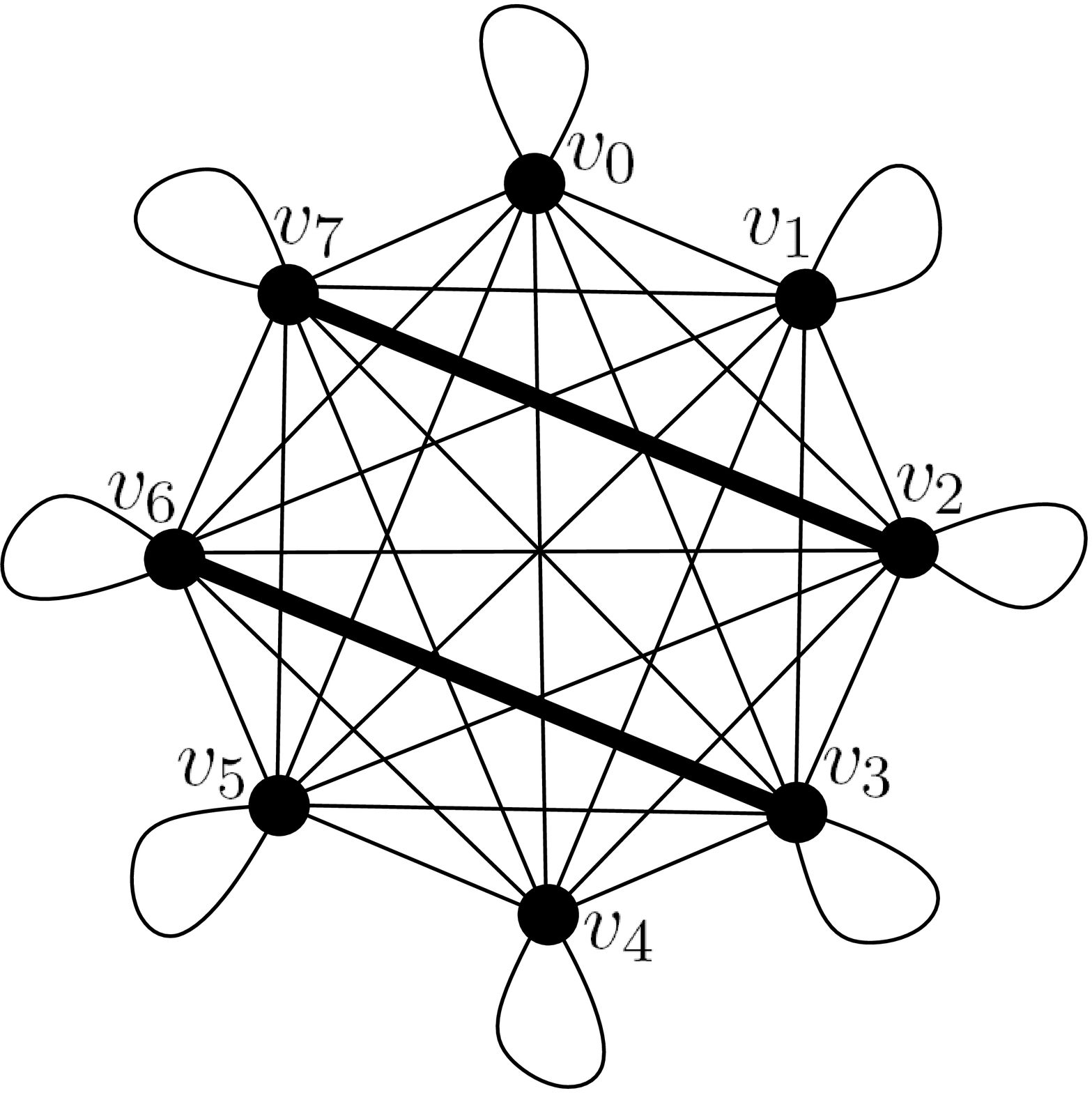}&
\includegraphics[width=0.25\textwidth]{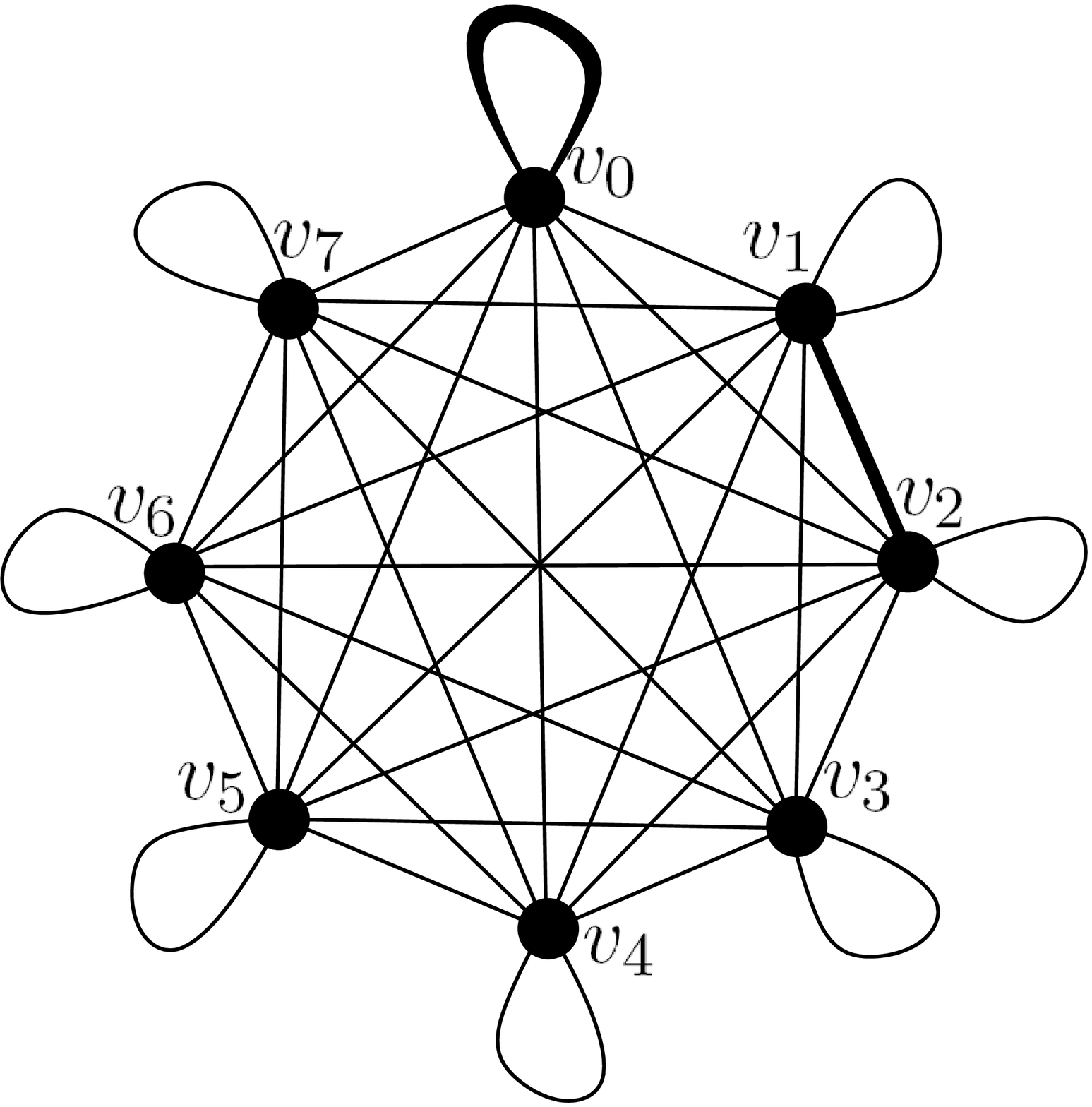}\\
(a) & (b) & (c)\\
\end{tabular}
\end{center}
\caption{Hypergraph $\mathcal{H}_{8,8}$. The selected edges form
the set: (a) $\mathcal{F}_0$, (b) $\mathcal{F}'_0$, (c)
$\mathcal{F}''_0$. }\label{finallemma2}
\end{figure}

\end{proof}

Now, we can formulate our main result in this section.

\begin{theorem}\label{n-k}
{\rm Det}$(K_{n+1:k})=n+1-k$ if and only if $K_{n+1:k}$ is
isomorphic to $K_{n+1:1}$, $K_{5:2}$ or $K_{6:2}$.

\end{theorem}

\begin{proof}
($\Longrightarrow$) By Lemma \ref{nk}, it does not exist a Kneser
graph $K_{n+1:k}$ verifying that $2k\leq n<\frac{k(k+1)}{2}$ and
{\rm Det}$(K_{n+1:k})=n+1-k$ since {\rm Det}$(K_{n+1:k})\leq
k<n+1-k$. Then, we can assume that $n\geq\frac{k(k+1)}{2}$ and
$k\geq 2$. Indeed, when $k=1$ the Kneser graph $K_{n+1:1}$ is
isomorphic to the complete graph $K_{n+1}$ and so {\rm
Det}$(K_{n+1:k})=n$.

Suppose first that there exists $d\in \mathbb{N}$ with $d>2$ such
that $n=\lfloor\frac{d(k+1)}{2}\rfloor$ . Then $d\geq k$ and by
Theorem \ref{discretes} we have {\rm Det}$(K_{n+1:k})=d$. Thus, it
suffices to prove that $d<\lfloor\frac{d(k+1)}{2}\rfloor+1-k$
except for $d=3$ and $k=2$ which is the graph $K_{5:2}$. Suppose
on the contrary that either $d\neq 3$ or $k\neq 2$ and $d\geq
\lfloor\frac{d(k+1)}{2}\rfloor+1-k$. We distinguish two cases.

\,

\emph{Case 1.} $d$ even or $k$ odd: The contradiction follows
since $(2-k-1)d\geq 2(1-k)$ and so $d\leq 2$.

\emph{Case 2.} $d$ odd and $k$ even: We have $2d\geq d(k+1)-2k+1$
what easily implies that $(k-1)(d-2)\leq 1$. Clearly, the
inequality only holds for $k=2$ and $d=3$.

\

Assume now that there exists $d\in \mathbb{N}$ with $3\leq
k+1\leq d$ such that $\lfloor\frac{(d-1)(k+1)}{2}\rfloor < n
<\lfloor\frac{d(k+1)}{2}\rfloor$. By Theorem \ref{gaps}, {\rm
Det}$(K_{n+1:k})=d$ and it suffices to show that $d<n+1-k$ except
for $k=2$ and $d=4$ what leads to the Kneser graph $K_{6:2}$. The
following expression holds for all positive integers $d,k$ and
$n$ satisfying the above conditions except for $k=2$ and $d=4$:
$$d-1<\lfloor\frac{(d-1)(k+1)}{2}\rfloor+1-k<n+1-k$$ Hence, the
result follows.

($\Longleftarrow$) The determining numbers of $K_{n+1:1}$,
$K_{5:2}$ and $K_{6:2}$ are $n$, 3 and 4 respectively.

\end{proof}

\section{Concluding Remarks}

We have introduced hypergraphs for finding determining sets of
Kneser graphs. This technique provides the determining number of
all Kneser graphs $K_{n:k}$ with $n\geq \frac{k(k+1)}{2}+1$. We
also show the usefulness of this approach by providing shorter
proofs (of those in \cite{boutin}) of the characterization of all
Kneser graphs with fixed determining number 2, 3 or 4, and
establishing those with fixed determining number 5. Finally, we
prove that it does not exists an infinite number of Kneser graphs
$K_{n,k}$ with $k\geq 2$ and determining number $n-k$, answering a
question posed by Boutin in \cite{boutin}.

It appears that our technique can also be applied to the values of
$n$ and $k$ in between the line $n=2k+1$ and the curve
$n=\frac{k(k+1)}{2}+1$ for which Det$(K_{n:k})$ remains to
compute (see Figure \ref{comparing}). Nevertheless, the values on
the line $n=2k+1$ with $n\neq 2^r-1$ would probably require
different arguments. We also believe that hypergraphs can be used
in the study of the determining number of other families of
graphs such as the Johnson graphs.

An interesting open problem is to find similar approaches to
compute other parameters related to graphs such as the metric
dimension. Perhaps hypergraphs can characterize not only
determining sets but also resolving sets.

\bibliographystyle{plain}

\begin{thebibliography}{1}

\bibitem{albertson2} M. O. Albertson and D. L. Boutin. Using determining
sets to distinguish Kneser graphs. \emph{Electron. J. Combin.},
14(1): Research paper 20 (Electronic), 2007.

\bibitem{albertson} M. O. Albertson and K. Collins. Symmetry breaking in graphs.
\emph{Electron. J. Combin.}, 3: Research paper 18 (Electronic),
1996.

\bibitem{barany} I. Bárány.
\newblock A short proof of Kneser's conjecture.
\newblock {\em J. Combinatorial Theory Ser. A}, 25: 325-326, 1978.


\bibitem{beineke} L. W. Beineke and R. J. Wilson. \emph{Graph connections:
Relationships between graph theory and other areas of
mathematics}. Volume 5 of \emph{Oxford Lecture Series in
Mathematics and its applications}. The Clarendon Press Oxford
University Press, New York, 1997.

\bibitem{berge} C. Berge.
\newblock {\em Hypergraphs. Combinatorics of finite sets}.
\newblock North- Holland Mathematical Library, 45.
North-Holland Publishing Co., Amsterdam, 1989.


\bibitem{boutin} D. Boutin.
\newblock Identifying graphs automorphisms using determining sets.
\newblock {\em Electron. J. Combin.}, 13(1): Research paper 78 (Electronic),
2006.

\bibitem{boutin2} D. L. Boutin. The determining number of a Cartesian product.
\emph{Journal of Graph Theory}, 61(2): 77--87, 2009.

\bibitem{CGPS} C. Cáceres, D. Garijo, M. L. Puertas and C. Seara.
\newblock On the determining number and the metric dimension of graphs.
\newblock {\em Electron. J. Combin.}, 17(1): Research paper 63 (Electronic),
2010.



\bibitem{chromaticgraphtheory} G. Chartrand and P. Zhang.
\newblock {\em Chromatic graph theory}.
\newblock Taylor and Francis, 2008.


\bibitem{GGL}
P. Duchet.
\newblock Hypergraphs.
\newblock {\em In Handbook of
Combinatorics (R. Graham, M. Gr\"{o}tschel, and L. Lov\'{a}sz,
Eds.)}, Vol. 1, pages: 381--432. Elsevier, Amsterdam, 1995.

\bibitem{erwin} D. Erwin and F. Harary. Destroying automorphisms by fixing nodes.
\emph{Discrete Math.}, 306: 3244--3252, 2006.

\bibitem{furedi} P. Frankl and Z. Füredi.
\newblock Extremal problems concerning Kneser graphs.
\newblock {\em J. Combinatorial Theory Ser. B}, 40: 270--284, 1986.

\bibitem{gibbons} C. R. Gibbons and J. D. Laison. Fixing numbers of graphs and
groups. \emph{Electron. J. Combin.}, 16: Research paper 39
(Electronic), 2009.

\bibitem{harary2} F. Harary. Survey of methods of automorphism
destruction in graphs. Invited address \emph{Eighth Quadrennial
International Conference on Graph Theory, Combinatorics,
Algorithms and Applications}, Kalamazoo, Michigan, 1996.


\bibitem{harary0} F. Harary. Methods of destroying the symmetries of a graph.
\emph{Bull. Malasyan Math. Sc. Soc.}, 24(2): 183--191, 2001.


\bibitem{mkneser} M. Kneser. Aufgabe 360.
\emph{Jahresbericht der Deutschen Mathematiker-Vereinigung 2.
Abteilung}, 58:27, 1955.

\bibitem{lovasz} L. Lóvasz. Kneser's conjecture, chromatic number and homotopy.
\emph{J. Combinatorial Theory Ser. A}, 25: 319--324, 1978.


\bibitem{lynch} K. Lynch. Determining the orientation of a painted sphere from a single image: a graph coloring problem.
URL: http://citeseer.nj.nec.com/469475.html.

\bibitem{matousek} J. Matou$\check{\rm s}$ek.
\newblock A combinatorial proof of Kneser's conjecture.
\newblock {\em Combinatorica}, 24(1): 163--170, 2004.

\bibitem{sarkaria} K. S. Sarkaria.
\newblock A generalized Kneser conjecture.
\newblock {\em J. Combinatorial Theory Ser. B}, 49: 236--240, 1990.


\bibitem{west} D. B. West.
\newblock {\em Introduction to Graph Theory}.
\newblock Prentice Hall, 1996.




\end{thebibliography}

\end{document}